\documentclass[11pt]{amsart}
\usepackage{mathrsfs}

\usepackage{amssymb}
\usepackage{amsfonts}

\usepackage{amscd}
\usepackage{bbm}
\usepackage{mathabx}
\usepackage{skak}
\usepackage{pigpen}

\usepackage[all]{xy}

\newtheorem{proposition}{Proposition}[section]
\newtheorem{statement}[proposition]{Statement}
\newtheorem{lemma}[proposition]{Lemma}
\newtheorem{definition}[proposition]{Definition}
\newtheorem{theorem}[proposition]{Theorem}

\newtheorem{corollary}[proposition]{Corollary}

\newtheorem{example}[proposition]{Example}
\newtheorem{remark}[proposition]{Remark}
\newtheorem{construction}[proposition]{Construction}
\newtheorem{lemma-definition}[proposition]{Lemma-Definition}

\newtheorem{question}[proposition]{Question}

\newcounter{tmp}

\def\coh{\operatorname{coh}}
\def\Qcoh{\operatorname{Qcoh}}

\def\lto{\longrightarrow}

\def\A{{\mathcal A}}
\def\B{{\mathcal B}}
\def\C{{\mathcal C}}

\def\D{{\mathcal D}}
\def\F{{\mathcal F}}
\def\E{{\mathcal E}}
\def\G{{\mathcal G}}
\def\H{{\mathcal H}}
\def\K{{\mathcal K}}
\def\N{{\mathcal N}}
\def\cO{{\mathcal O}}
\def\R{{\mathcal R}}
\def\L{{\mathcal L}}

\def\M{{\mathcal M}}
\def\cS{{\mathcal S}}

\def\T{{\mathcal T}}
\def\I{{\mathcal I}}

\def\ZZ{{\mathbb Z}}
\def\CC{{\mathbb C}}

\def\bR{{\mathbf R}}

\def\bp{{\mathbf p}}
\def\bL{{\mathbf L}}

\def\NN{{\mathbb N}}
\def\ZZ{{\mathbb Z}}

\def\CC{{\mathbb C}}

\def\PP{{\mathbb P}}

\def\Hom{\operatorname{Hom}}
\def\End{\operatorname{End}}
\def\Ext{\operatorname{Ext}}

\def\Spec{\operatorname{Spec}}

\def\id{{\operatorname{id}}}

\def\kk{{\mathbf k}}

\def\op{\circ}
\def\pr{\operatorname{pr}}

\def\prd{\boxtimes}

\newcommand{\Ho}{{\H^0}}

\newcommand{\Ob}{\operatorname{Ob}}

\newcommand{\SF}{\dS\!\dF\!\operatorname{--}\!}
\newcommand{\SFf}{\dS\!\dF_{fg}\!\operatorname{--}\!}
\newcommand{\prfdg}{\mathscr{P}\!\mathit{erf}\!\operatorname{--}}
\newcommand{\dRep}{\mathscr{R}\!\mathit{ep}}

\newcommand{\Ac}{\dA\!\mathit{c}\!\operatorname{--}\!}

\newcommand{\dCom}{\dC\!\mathit{om}\!\operatorname{--}\!}
\newcommand{\dAc}{\dA\!\mathit{c}\!\operatorname{--}\!}

\newcommand{\prf}{\mathcal{P}\!\mathit{erf}\!\operatorname{--}}

\newcommand{\dHf}{\dF\!\mathit{lat}\!\operatorname{--}\!}
\newcommand{\dHAc}{\dA\!\mathit{c}_{f}\!\operatorname{--}\!}

\def\dA{\mathscr A}
\def\dB{\mathscr B}
\def\dC{\mathscr C}
\def\dD{\mathscr D}
\def\dE{\mathscr E}
\def\dT{\mathscr T}

\def\dF{\mathscr F}

\def\dS{\mathscr S}
\def\dM{\mathscr M}
\def\dN{\mathscr N}
\def\dP{\mathscr P}
\def\dR{\mathscr R}
\def\dI{\mathscr I}

\def\Mod{{\mathscr M}\!\mathit{od}\!\operatorname{--}\!}

\def\mC{\mathsf C}
\def\mM{\mathsf M}
\def\mE{\mathsf E}
\def\mG{\mathsf G}
\def\mF{\mathsf F}
\def\mN{\mathsf N}
\def\mP{\mathsf P}
\def\mT{\mathsf T}
\def\mS{\mathsf S}

\def\mU{\mathsf U}
\def\mV{\mathsf V}

\def\ma{\mathsf a}
\def\mb{\mathsf b}
\def\me{\mathsf e}
\def\mh{\mathsf h}
\def\mf{\mathsf f}
\def\mi{\mathsf i}
\def\mpr{\mathsf{pr}}
\def\mp{\mathsf{p}}

\def\mPhi{\mathsf \Phi}
\def\mPsi{\mathsf \Psi}
\def\mpi{\mathsf \pi}
\def\mtau{\mathsf \tau}

\def\mX{X}
\def\mY{Y}
\def\mZ{Z}

\def\dHom{\mathsf{Hom}}

\def\tr{\mathcal{T}\!\mathit{r}}

\def\ptr{\operatorname{pre-tr}}

\def\lHom{\underline{\H\!\mathit{om}}}

\def\wt{\widetilde}

\def\Tor{\T\mathit{or}}

\def\DGcat{{\D\!\G\!\mathit{cat}}}
\def\Hqe{{\H\!\mathit{qe}}}
\def\Rep{{\R\!\mathit{ep}}}

\def\Tot{\operatorname{Tot}}

\def\RPNS{\mathbf{NSch_{sm}^{pr}}}
\def\GNS{\mathbf{GNSch}}

\def\rd{\mathfrak{R}}
\def\md{\operatorname{mod}\!\operatorname{--}\!}
\def\Md{\operatorname{Mod}\!\operatorname{--}\!}

\def\gldim{\operatorname{gl.dim}\;}

\newenvironment{dok}{\par\vspace{-5pt}%
\par\noindent\begingroup%
\leftskip=0em\hspace{0em}{\bf Proof.}}%
{\endgroup\hfill$\Box$}

{\endgroup\hfill$\Box$}

\def\hy{\mbox{-}}

\def\llto{{\;\relbar\joinrel\relbar\joinrel\relbar\joinrel\relbar\joinrel
\rightarrow\;}}

\title[]{Smooth and proper noncommutative schemes and\\ gluing of  DG categories}

\author[]{Dmitri Orlov}

\address{ Algebraic Geometry Section, Steklov Math. Institute RAS,
8 Gubkin str., Moscow 119991, RUSSIA}
\email{orlov@mi.ras.ru}

\thanks{This work is supported by the Russian Science Foundation under  grant 14-50-00005}

\date{}
\dedicatory{To my wife Maria with thankfulness and  love}


\begin{document}
\begin{abstract}
In this paper we discuss different properties of noncommutative schemes over a field.
We define a noncommutative scheme as a differential graded category of a special type.
We study regularity, smoothness and properness for noncommutative schemes.
Admissible subcategories
of categories of perfect complexes on smooth projective schemes provide natural examples of  smooth and proper noncommutative schemes
that are called geometric noncommutative schemes.
In this paper we show that the world of all geometric noncommutative schemes is closed
under an operation of a gluing of differential graded categories via  bimodules.
As a consequence of the main theorem we obtain that
for any finite dimensional algebra with separable semisimple part
the category of perfect complexes over it is equivalent to a full subcategory of the category of perfect complexes on
a smooth projective scheme.
Moreover, if the algebra has  finite global dimension, then
the full subcategory is  admissible.
We also provide a construction of a smooth projective scheme
that admits a full exceptional collection and contains
as a subcollection an exceptional collection  given in advance.
As another application of the main theorem we obtain, in characteristic 0, an existence of a full embedding for the category of perfect complexes
on any proper scheme to the category of perfect complexes on a smooth projective scheme.
\end{abstract}

\maketitle


\section*{Introduction}

One of the main approaches to noncommutative geometry is to consider categories of sheaves on varieties instead of  varieties
themselves. In algebraic geometry only quasi-coherent sheaves represent well the algebraic structure of a variety and
do not depend on a topology. Besides, homological algebra convinces to consider derived categories whenever we meet an abelian category.
Thus, instead of schemes one  may consider  derived versions of categories of quasi-coherent sheaves and triangulated category
of perfect complexes that are compact objects therein. This approach is very powerful.

Let $X$ be a scheme over a field $\kk.$ Studying schemes it is natural to consider the unbounded
derived category of quasi-coherent sheaves $\D(\Qcoh X)$ and the unbounded derived category of complexes of $\cO_X$\!--modules with
quasi-coherent cohomology $\D_{\Qcoh}(X).$ Fortunately, and it is well known and proved in \cite{BN},  for  a quasi-compact and separated scheme $X$
the canonical functor $\D(\Qcoh X)\to \D_{\Qcoh}(X)$ is an equivalence. Moreover, it was shown in \cite{Ne}
that in this case the derived category $\D(\Qcoh X)$ has enough compact objects and the subcategory
of compact objects is nothing else but the subcategory of perfect complexes $\prf X.$
Recall that a complex is called perfect if it is locally quasi-isomorphic to a bounded complex
of locally free sheaves of finite type.

Furthermore, it was proved in \cite{Ne, BVdB} that the category $\prf X$ admits a classical generator $\mE,$ i.e.
the minimal full triangulated subcategory of $\prf X$ that contains $\mE$ and is closed under direct summands
coincides with the whole $\prf X.$ As a consequence, such a perfect complex $\mE$ is
a compact generator of the whole $\D(\Qcoh X).$

It is very useful to consider a triangulated category $\T$ together with
an {\em enhancement} $\dA,$ which is a differential graded (DG) category
with the same objects as in $\T,$ but
the set of morphisms between two objects in $\dA$ is a complex of vector spaces.
One recovers morphisms in $\T$ by taking the cohomology $H^0$ of the
corresponding morphism complex in $\dA.$ Thus,
$\T$ is the {\em homotopy} category $\Ho (\dA)$ of the DG category $\dA.$
Such an $\dA$ is called an {\em enhancement} of $\T.$

The triangulated category $\D(\Qcoh X)$ has several natural DG enhancements:
the category of h-injective complexes, the DG quotient of all complexes by acyclic complexes, the DG quotient
of h-flat complexes by acyclic h-flat complexes. They all are quasi-equivalent and we can work with any of them.
Denote by $\dD(\Qcoh X)$ a DG enhancement of $\D(\Qcoh X)$ and by $\prfdg X$ the induced
DG enhancement of the category of perfect complexes $\prf X.$

Let us take a generator $\mE\in\prfdg X.$ Denote by $\dE$ its DG algebra of endomorphisms,
i.e. $\dE=\dHom(\mE, \mE).$ Since $\mE$ is perfect the DG algebra $\dE$ has only finitely many nonzero cohomology groups.
Keller's results from \cite{Ke} imply that
the DG category  $\prfdg X$ is quasi-equivalent to
$\prfdg \dE,$ where $\dE$ is a cohomologically bounded DG algebra and $\D(\Qcoh X)$ is equivalent
to the derived category of DG $\dE$\!--modules $\D(\dE).$

The facts described above allow us to suggest a definition of a {\em (derived) noncommutative scheme} over $\kk$
as a $\kk$\!--linear DG category of the form $\prfdg\dE,$
where $\dE$ is a cohomologically  bounded DG algebra over $\kk.$
In this case the derived category
$\D(\dE)$ will be called the derived category of quasi-coherent sheaves on this noncommutative scheme.

The simplest and, apparently, the most important class of schemes is the class of smooth and projective schemes
or, more generally, the class of schemes that are regular and proper.
The properties of regularity, smoothness and properness can be interpreted in categorical terms and
can be extended to noncommutative schemes.
We say that  a $\kk$\!--linear triangulated category $\T$ is {\em proper} if
$\bigoplus_{m\in\ZZ}\Hom(X, Y[m])$ is finite dimensional for any two objects $X, Y\in\T.$
It will be called {\em regular} if it has a strong generator, i.e. such an object that generates the whole $\T$
 in a finite number of steps
(Definitions \ref{strong_gen} and \ref{reg_and_prop}).
The notion of smoothness is well-defined for a DG category. A $\kk$\!--linear DG category $\dA$ is called {\em $\kk$\!-smooth} if it is perfect as a module over  $\dA^{\op}\otimes_{\kk}\dA.$
Application of these definitions to a DG category of perfect objects $\prfdg\dE$ and its homotopy category
 $\prf\dE$ leads  us to well-defined notions of a
 regular, smooth, and proper  noncommutative schemes.

It can be proved that a usual separated noetherian scheme $X$ is regular if and only if the category
$\prf X$ is regular (Theorem \ref{regular}). Furthermore, for separated schemes of finite type
properness of $X$ is equivalent to properness of $\prf X$
(Propositions \ref{prop_scheme}) and
smoothness of $X$ is equivalent of $\prfdg X$
(\cite{Lu} and Propositions \ref{smooth_proper}).
These facts imply that smooth and projective schemes form a subclass of
the class of smooth and proper noncommutative schemes. It is  not difficult to give an example of a noncommutative
smooth and proper scheme $\prfdg\dE$ that is not quasi-equivalent to $\prfdg X$ of a usual commutative scheme.

Let us consider the world of all smooth projective schemes.
As in the theory of motives, an important step is adding direct summands,
in our situation it is  natural to extend the world of smooth projective schemes to the world of all admissible subcategories
$\N\subset \prf X,$ where $X$ is a smooth and projective scheme.
(Recall that a full triangulated subcategory $\N\subset\prf X$ is called admissible if the inclusion
functor has right and left adjoint functors and, hence, $\N$ is a semi-orthogonal summand of $\prf X.$)
Such admissible subcategories are also smooth and proper, and they give natural examples of smooth and proper noncommutative schemes, which will be called {\em geometric noncommutative schemes}.

One may consider the 2-category of smooth and proper noncommutative schemes
$\RPNS$ over a field $\kk.$
Objects of $\RPNS$ are DG categories $\dA$ of the form $\prfdg\dE,$ where
$\dE$ is a smooth and proper DG algebra; 1-morphisms are
quasi-functors, i.e. DG functors modulo inverting quasi-equivalences; and   2-morphisms are morphisms of quasi-functors.
The 2-category $\RPNS$ has a full 2-subcategory of geometric smooth and proper noncommutative schemes $\GNS.$
Evidently, $\GNS$ contains all smooth and projective commutative schemes. Moreover, a To\"en's theorem \cite{To}
says  that quasi-functors from $\prfdg X$ to $\prfdg Y$ correspond bijectively to perfect complexes on the product, i.e.
 $\prf(X\times Y)$ is the category of morphisms
from $X$ to $Y$ in  $\RPNS.$

The world of smooth and proper noncommutative schemes is plentiful and multiform in the sense of different constructions and operations.
For instance, it contains all $\prfdg \Lambda$ for all finite dimensional algebras $\Lambda$ of finite global dimension.
Besides, for any two noncommutative schemes $\dA$ and $\dB$  every perfect $\dB^{\op}\hy\dA$\!--bimodule
$\mS$ produces a new noncommutative scheme $\dC=\dA\underset{\mS}{\oright}\dB,$ which we call the gluing of
 $\dA$ and $\dB$ via $\mS.$ The resulting DG category $\dC$ is also smooth and proper (Definition
\ref{gluing_cat} and Section \ref{reg_proper_noncommutative}).
Its homotopy category has a semi-orthogonal decomposition of the form $\Ho(\dC)=\langle\Ho(\dA),\, \Ho(\dB)\rangle.$
Of course, this procedure can be iterated and it allows to reproduce new and new noncommutative schemes.
If we glue  commutative schemes like $\prfdg X$ and $\prfdg Y,$ then the result is almost never a commutative scheme, except
for a few special examples (Examples \ref{ex1} and \ref{ex2}).

The main purpose of this paper is to show that the world of all geometric noncommutative schemes is closed
under operation of the gluing. More precisely, we prove that for any smooth and projective
$X$ and $Y$  the gluing $\prfdg X\underset{\mS}{\oright}\prfdg Y$
of two DG categories of the form $\prfdg X$ and $\prfdg Y$ via arbitrary perfect bimodule $\mS$
is a geometric noncommutative scheme, i.e. it is quasi-equivalent to an admissible full DG subcategory in $\prfdg V$
for some smooth projective scheme $V$ (Theorem \ref{main}).
This result implies that the subcategory of smooth proper geometric noncommutative schemes is closed under gluing
via any bimodules (Theorem \ref{main2}).

These theorems have useful applications. Using results of \cite{KL} we obtain that
 for any proper scheme
$Y$ over a field of characteristic $0$ there is a full embedding of $\prf Y$ into $\prf V,$ where $V$ is smooth and projective (Corollary \ref{emb}).
In Section \ref{applications} we show that for any finite dimensional algebra $\Lambda$
with the semisimple part $S=\Lambda/\rd$ separable over the base field $\kk,$ there exist a smooth projective scheme
$V$ and a perfect complex $\E^{\cdot}$ on $X$ such that $\End(\E^{\cdot})=\Lambda$ and
$\Hom(\E^{\cdot}, \E^{\cdot}[l])=0$ when $l\ne 0$ (Theorem \ref{algebra}).
As a consequence of this theorem we obtain that
for any finite dimensional algebra $\Lambda$ over $\kk$ with separable semisimple part $S=\Lambda/\rd$
there is a smooth projective scheme $V$ such that
the DG category $\prfdg\Lambda$ is quasi-equivalent to a full DG subcategory of $\prfdg V.$
Moreover, if $\Lambda$ has  finite global dimension, then
$\prf\Lambda$ is  admissible in $\prf V$ (Corollary \ref{algebra_inclusion}).
Note that over a perfect field all algebras a separable.

In Section \ref{exceptional_coll} we give an alternative and  more useful procedure of constructing a smooth projective scheme
that admits a full exceptional collection and contains
as a subcollection an exceptional collection  given in advance.
More precisely, for any DG category $\dA,$ for which  the homotopy category
$\Ho(\dA)$ has a full exceptional collection,
we give an explicit construction of a smooth projective scheme $X$ and an exceptional collection of line bundles
$\sigma=(\L_1,\dots, \L_n)$ in $\prf X$ such that the DG subcategory $\dN\subset\prfdg X$
generated by $\sigma$ is quasi-equivalent to $\dA.$
Moreover, by construction $X$ is rational and has a full exceptional collection (Theorem \ref{exc_col}).

In the last section we illustrate this theorem considering the case
of noncommutative projective planes, in the sense of noncommutative deformations of the usual projective plane,
which have been introduced and described by M.~Artin, J.~Tate, and M.~Van den Bergh in \cite{ATV}.

The author is very grateful to Valery Alexeev, Alexei Bondal, Sergei Gorchinskiy,
Anton Fonarev, Alexander Kuznetsov, Valery Lunts, Amnon Neeman, Stefan Nemirovski, Yuri Prokhorov,
Constantin Shramov for useful and valuable discussions.
The author would like to thank anonymous referee and
Aise Johan de Jong for  pointing out the results of the publication \cite{LN},
which allowed to prove Proposition \ref{prop_scheme} in full generality.
The author wishes to express his gratitude to Theo Raedschelders who drew attention
to the results of Dieter Happel's paper \cite{Hap}.

\section{Preliminaries on triangulated categories, generators, and semi-orthogonal decompositions }

\subsection{Generators in triangulated categories}

In this section we discuss different notions of generators in triangulated categories.
Let $\T$ be a triangulated category and $S$ be a set of objects.

\begin{definition}
A set of objects $S\subset \Ob\T$
{\em generates} the triangulated category $\T$ if $\T$ coincides with the smallest strictly full triangulated subcategory of
$\T$ which contains $S.$ (Strictly full means it is full and closed under isomorphisms).
\end{definition}

The notion of generating a triangulated category is very rigid, because a triangulated subcategory
that is generated by a set of objects is not necessarily idempotent complete. A much more useful notion of generating
a triangulated category is the notion of a set of classical generators.

\begin{definition}
A set of objects $S\subset \Ob\T$
forms a {\em set of classical generators} for $\T$ if the category $\T$ coincides with the smallest triangulated subcategory of
$\T$ which contains $S$ and is closed under taking direct summands.
When $S$ consists of a single object we obtain the notion of a {\em classical generator}.
\end{definition}
If a classical generator $X$ generates the whole category in a finite number of steps, then it is called
a {\em strong generator}.
More precisely, let $\I_1, \I_2 \subset
\mathcal{T}$ be two full subcategories. Define their product as
\[ \I_1 * \I_2 = \bigl\{ \text{the full subcategory, consisting on all objects $Y$ of the form} : Y_1 \to Y \to Y_2 \text{ , } Y_i \in \I_i \bigr\} \text{ .} \]
Let $\langle \I \rangle$ be the smallest full subcategory
that contains $\I \subset \langle \I\rangle$ and that is
closed under shifts, finite direct sums, and direct summands. We call
$\langle \I \rangle$ the \emph{envelope} of $\I.$

Put
$\I_1 \diamond \I_2 := \langle \I_1 *
\I_2 \rangle$
and we define by induction
$\langle \I\rangle_k=\langle\I\rangle_{k-1}\diamond\langle \I\rangle.$ If $\I$ consists of a single object $X$ we denote $\langle \I\rangle$ by
 $\langle X\rangle_1$ and put by induction $\langle X\rangle _k=\langle X\rangle_{k-1}\diamond\langle X\rangle_1.$

\begin{definition}\label{strong_gen}
An object $X$ is called a {\em strong generator} if $\langle X\rangle_n=\T$ for some $n\in\NN.$
\end{definition}

\begin{remark}{\rm
If $X \in \mathcal{T}$ is a classical generator, then $\mathcal{T} =
\bigcup_{i=1}^\infty \langle X \rangle_i.$
It is also easy to see that if $\mathcal T$ has a strong generator,
then any classical generator is strong as well.}
\end{remark}

Following \cite{Ro} we define the dimension of a triangulated
category.
\begin{definition}
The {\em dimension} of a triangulated category $\T,$ denoted by $\dim \T,$ is the smallest integer $d\ge 0$ such that there exists an object $X\in \T$
for which $\langle X\rangle_{d+1}=\T.$
\end{definition}

Let now $\T$ be a triangulated category which admits arbitrary small coproducts (direct sums).
Such a category is called {\em cocomplete.}

\begin{definition} Let $\T$ be a cocomplete triangulated category.
An object $X\in \T$ is called {\em compact} in $\T$ if $\Hom(X,-)$ commutes with arbitrary small coproducts, i.e. for any
set of objects $\{ Y_i\}\subset \T$ the canonical map
$
\bigoplus_i\Hom (X, Y_i){\longrightarrow}
\Hom (X, \bigoplus_i Y_i)
$ is an isomorphism.
\end{definition}
Compact objects in $\T$ form a triangulated subcategory denoted by $\T^c\subset \T.$

\begin{definition}\label{van}
 Let $\T$ be a cocomplete triangulated category. A set $S \subset \Ob\T^c$ is
called a {\em set of compact generators} if
any object $Y\in \T$ for which $\Hom(X, Y[n])=0$ for all $X\in S$ and  all $n\in \ZZ$ is a zero object.
\end{definition}
\begin{remark}
{\rm
Since $\T$ is cocomplete, it can be proved that the property of $S \subset \Ob\T^c$ to be
 a set of compact generators is  equivalent to the following property:
the category $\T$ coincides with the smallest full
triangulated subcategory containing $S$ and  closed under
small coproducts \cite{Ne1}.
}
\end{remark}

\begin{remark}\label{twodef}{\rm
The definition of compact generators is  closely related to the definition of classical generators.
Assume that a cocomplete triangulated category $\T$  is compactly generated by the set
of compact objects $\T^c.$ In this case a set $S\subset \T^c$ is a set of compact generators of $\T$ if and only if
the set $S$ is a set of classical generators of the subcategory of compact objects $\T^c$ \cite{Ne1}.
}
\end{remark}

Let  $\T$ be  a cocomplete triangulated category and let
$X\in\T^c$ be a compact object. If on each step we add not only finite sums but also all arbitrary direct sums
one can define full subcategories $\widebar{\langle X\rangle}_k\subset\T.$
The following proposition is proved in \cite[2.2.4]{BVdB}.

\begin{proposition}\label{compact_gen}
If $X$ is a compact object in a cocomplete triangulated category $\T,$ then
$\widebar{\langle X\rangle}_k\bigcap\T^c=\langle X\rangle_k.$
\end{proposition}

In particular, if $\widebar{\langle X\rangle}_k=\T$ for some $k,$ then $\langle X\rangle_k=\T^c$ and
$X$ is a strong generator of $\T^c.$

\subsection{Semi-orthogonal decompositions}

Let $\T$ be a $\kk$\!--linear triangulated category, where $\kk$ is a base field.
Recall some definitions and facts concerning admissible
subcategories and semi-orthogonal decompositions.
Let ${\N\subset\T}$ be a full triangulated subcategory.  The {\em
right orthogonal} to ${\N}$ is the full subcategory
${\N}^{\perp}\subset {\T}$ consisting of all objects $X$ such that
${\Hom(Y, X)}=0$ for any $Y\in{\N}.$ The {\em left orthogonal}
${}^{\perp}{\N}$ is defined analogously.  The orthogonals are also
triangulated subcategories.

\begin{definition}\label{adm}
Let $I\colon\N\hookrightarrow\T$ be a full embedding of triangulated
categories. We say that ${\N}$
is {\em right admissible} (respectively {\em left admissible}) if
there is a right (respectively left) adjoint functor $Q\colon\T\to \N.$ The
subcategory $\N$ will be called {\em admissible} if it is both right and left
admissible.
\end{definition}
\begin{remark}\label{semad}
{\rm The subcategory $\N$ is right admissible
if and only if for each object $Z\in{\T}$ there is an
exact triangle $Y\to Z\to X,$ with $Y\in{\N},\, X\in{\N}^{\perp}.$ }
\end{remark}

Let $\N$ be a full triangulated subcategory in a triangulated category
$\T.$ If $\N$ is right (respectively left) admissible, then the
quotient category $\T/\N$ is equivalent to $\N^{\perp}$ (respectively~${}^{\perp}\N$).  Conversely, if the quotient functor $Q\colon\T\lto\T/\N$
has a left (respectively right) adjoint, then~$\T/\N$ is equivalent to
$\N^{\perp}$ (respectively~${}^{\perp}\N$).

\begin{definition}\label{sd}
A {\em semi-orthogonal decomposition} of a triangulated category $\T$
is a sequence of full triangulated subcategories ${\N}_1, \dots, {\N}_n$ in ${\T}$
such that there is an increasing filtration $0=\T_0\subset\T_1\subset\cdots\subset\T_n=\T$
by left admissible subcategories for which the left orthogonals ${}^{\perp}\T_{i-1}$
in $\T_{i}$ coincide with $\N_i.$ In particular, $\N_i\cong\T_i/\T_{i-1}.$
We write
$
{\T}=\left\langle{\N}_1, \dots, {\N}_n\right\rangle.
$
\end{definition}

If we have a semi-orthogonal decomposition ${\T}=\left\langle{\N}_1, \dots, {\N}_n\right\rangle,$ then the inclusion functors induce an isomorphism of the Grothendieck groups
\[
K_0(\N_1)\oplus K_0(\N_2)\oplus\cdots\oplus K_0(\N_n)\cong K_0(\T).
\]

It is more convenient to consider so called enhanced triangulated categories, i.e. triangulated categories
that are  homotopy categories  of pretriangulated DG categories (see Section \ref{enhancements}).
An enhancement of a triangulated category  $\T$ induces an enhancement of any full triangulated subcategory
$\N\subset\T.$
Using an enhancement of a triangulated category $\T$ we can define the K--theory spectrum $K(\T)$ of $\T$
(see \cite{Ke2}).  It also gives us an additive invariant (see, for example, \cite[5.1]{Ke2}), i.e for any
semi-orthogonal decomposition we have an isomorphism
\[
K_*(\N_1)\oplus K_*(\N_2)\oplus\cdots\oplus K_*(\N_n)\cong K_*(\T).
\]

\subsection{Exceptional, w-exceptional, and semi-exceptional collections}

The existence of a semi-orthogonal decomposition on a triangulated
category $\T$ clarifies the structure of $\T.$ In the best scenario,
one can hope that $\T$ has a semi-orthogonal decomposition
${\T}=\left\langle{\N}_1, \dots, {\N}_n\right\rangle$ in which each
$\N_p$ is as simple as possible, i.e. is
equivalent to the bounded derived category of finite-dimensional
vector spaces.

\begin{definition}\label{exc}
An object $E$ of a $\kk$\!--linear triangulated category ${\T}$ is
called {\em exceptional} if  ${\Hom}(E, E[l])=0$ whenever $l\ne 0,$ and
${\Hom}(E, E)=\kk.$ An {\em exceptional collection} in ${\T}$ is a
sequence of exceptional objects $\sigma=(E_1,\dots, E_n)$ satisfying the
semi-orthogonality condition ${\Hom}(E_i, E_j[l])=0$ for all $l$ whenever
$i>j.$
\end{definition}

If a triangulated category $\T$ has an exceptional collection
$\sigma=(E_1,\dots, E_n)$ that generates the whole of $\T,$ then
this collection is called {\em full}.  In this case $\T$ has a
semi-orthogonal decomposition with $\N_p=\langle E_p\rangle.$ Since
$E_{p}$ is exceptional, each of these categories is equivalent to the
bounded derived category of finite dimensional $\kk$\!-vector spaces.  In
this case we write $ \T=\langle E_1,\dots, E_n \rangle.$

\begin{definition}\label{strong}
An exceptional collection $\sigma=(E_1,\dots, E_n)$ is called {\em strong} if, in addition,
${\Hom}(E_i, E_j[l])=0$ for all  $i$ and $j$ when $l\ne 0.$
\end{definition}

The best known example of an exceptional collection is the sequence
of invertible sheaves
$\left(\mathcal{O}_{\mathbb{P}^n},\dots,\mathcal{O}_{\mathbb{P}^n}(n)\right)$
on the projective space $\mathbb{P}^n.$ This exceptional collection is
full and strong.

When the field $\kk$ is not algebraically closed it is reasonable to weaken the notions of an exceptional
object and an exceptional collection.
\begin{definition}\label{semi-exc}
An object $E$ of a $\kk$\!--linear triangulated category ${\T}$ is
called {\em w-exceptional (weak exceptional)} if  ${\Hom}(E, E[l])=0$ when $l\ne 0,$ and
${\Hom}(E, E)=D,$ where $D$ is a finite dimensional division algebra over $\kk.$
It is called {\em semi-exceptional} if ${\Hom}(E, E[l])=0$ when $l\ne 0$ and
${\Hom}(E, E)=S,$ where $S$ is a finite dimensional semisimple algebra over $\kk.$
\end{definition}

It is evident that exceptional and semi-exceptional objects are stable under base field change while
w-exceptional objects are not.

A {\em w-exceptional (semi-exceptional) collection} in ${\T}$ is a
sequence of w-exceptional (semi-exceptional) objects $(E_1,\dots, E_n)$ with
semi-orthogonality conditions ${\Hom}(E_i, E_j[l])=0$ for all $l$ whenever
$i>j.$

\begin{example}\label{Severi-Brauer}
{\rm
Let $\kk$ be a field and $D$ be a central simple algebra over $\kk.$
Consider a Severi-Brauer variety $SB(D).$  There is a full semi-exceptional collection
$(S_0, S_1,\dots S_n)$ on $SB(D)$ such that $S_0=\cO_{SB}$ and
$\End(S_i)\cong D^{\otimes i},$ where $n+1$ is the order of the class of $D$ in the Brauer group of $\kk.$
 Since each $D^{\otimes i}$ is a matrix algebra over a central division algebra $D_i,$ there is a w-exceptional collection
$(E_0, E_1,\dots, E_n)$ such that $\End(E_i)\cong D_i$ (see \cite{Be} for a proof).
In this situation $S_i$ are isomorphic to $E_i^{\oplus k_i}$ for some integers $k_i.$
These collections are also strong.
}
\end{example}

\section{Preliminaries on differential graded categories}

\subsection{Differential graded categories}

Our main references for differential graded (DG) categories are \cite{Ke,Dr,To, TV}.
Here we only recall some points and introduce notation.
Let $\kk$ be a field.
All categories, DG categories, functors, DG
functors and so on are assumed to be $\kk$\!--linear.

A {\it differential graded or DG category} is a $\kk$\!--linear category $\dA$ whose morphism spaces $\dHom (\mX, \mY)$
are complexes of $\kk$\!-vector spaces (DG $\kk$\!--modules), so that for any $\mX, \mY, \mZ\in
\Ob\dC$ the composition $\dHom (\mY, \mZ)\otimes \dHom (\mX, \mY)\to
\dHom (\mX, \mZ)$ is a morphism of DG $\kk$\!--modules. The identity morphism $1_\mX\in \dHom (\mX, \mX)$ is closed of
degree zero.

Using the supercommutativity isomorphism $\mU\otimes \mV\simeq \mV\otimes
\mU$ in the category of DG $\kk$\!--modules one defines for every DG
category $\dA$ the {\it opposite DG category} $\dA^{\op}$ with $\Ob\dA
^{\op}=\Ob\dA$ and $\dHom_{\dA^{\op}}(\mX, \mY)=\dHom_{\dA}(\mY, \mX).$

For a DG category $\dA$ we denote by $\Ho(\dA)$
its homotopy category.
The {\it homotopy category} $\Ho(\dA)$ has the same objects as the DG category $\dA$ and its
morphisms are defined by taking the $0$\!-th cohomology
$H^0(\dHom_{\dA} (\mX, \mY))$
of the complex $\dHom_{\dA} (\mX, \mY).$

As usual, a {\it DG functor}
$\mF:\dA\to\dB$ is given by a map $\mF:\Ob(\dA)\to\Ob(\dB)$ and
by morphisms of DG $\kk$\!--modules
$$
\mF_{\mX, \mY}: \dHom_{\dA}(\mX, \mY) \lto \dHom_{\dB}(\mF \mX,\mF \mY),\quad \mX, \mY\in\Ob(\dA)
$$
compatible with the composition and the units.

A DG functor $\mF: \dA\to\dB$ is called a {\it quasi-equivalence} if
$\mF_{\mX, \mY}$ is a quasi-isomorphism for all pairs of objects $\mX, \mY$ of $\dA$
and the induced functor $H^0(\mF): \Ho(\dA)\to \Ho(\dB)$ is an
equivalence. DG categories $\dA$ and $\dB$ are called {\it quasi-equivalent} if there exist DG
categories $\dC_1,\dots, \dC_n$ and a chain of quasi-equivalences
$\dA\stackrel{\sim}{\leftarrow} \dC_1 \stackrel{\sim}{\rightarrow} \cdots \stackrel{\sim}{\leftarrow} \dC_n
\stackrel{\sim}{\rightarrow} \dB.$

\subsection{Differential graded modules}

Given a small DG category $\dA$ we define a {\it right DG $\dA$\!--module} as a DG functor
$\mM: \dA^{op}\to \Mod \kk,$ where $\Mod \kk$ is the DG category of DG $\kk$\!--modules. We denote by $\Mod \dA$ the DG
category of right DG $\dA$\!--modules.

Each object $\mY$ of $\dA$ produces a right module
represented by $\mY$
\[
\mh^\mY(-):=\dHom_{\dA}(-, \mY)
\]
which is called a {\it representable} DG module. This
gives the Yoneda DG functor
$\mh^\bullet :\dA \to
\Mod\dA$ that is full and faithful.

A DG $\dA$\!--module is called {\it free} if it is isomorphic to a direct sum of  DG modules of the form
$\mh^\mY[n],$ where $\mY\in\dA,\; n\in\ZZ.$
A DG $\dA$\!--module
$\mP$ is called {\it semi-free} if it has a filtration
$0=\mPhi_0\subset \mPhi_1\subset ...=\mP$
such that each quotient  $\mPhi_{i+1}/\mPhi_i$ is free. The full
DG subcategory of semi-free DG modules is denoted by $\SF\dA.$
We denote by $\SFf\dA\subset \SF\dA$ the full DG subcategory of finitely generated semi-free
DG modules, i.e. such that $\mPhi_m=\mP$ for some $m$ and $\mPhi_{i+1}/\mPhi_i$ is a finite direct sum of DG modules of the form
$\mh^Y[n].$

For every DG
$\dA$\!--module $\mM$ there is a quasi-isomorphism $\bp \mM\to \mM$ such that $\bp M$ is a semi-free
DG $\dA$\!--module (see \cite{Ke} 3.1, \cite{Hi} 2.2, \cite{Dr} 13.2).

Denote by $\Ac\dA$ the full
DG subcategory of $\Mod\dA$ consisting of all acyclic DG modules, i.e. DG modules $\mM$
for which the complex $\mM(\mX)$ is acyclic for all $X\in\dA.$
The
homotopy category of DG modules $\Ho(\Mod\dA)$ has a natural structure of a triangulated category
and the homotopy subcategory of acyclic complexes $\Ho (\Ac\dA)$ forms a full triangulated subcategory in it.
The {\it derived
category} $\D(\dA)$ is defined as the Verdier quotient
\[
\D(\dA):=\Ho(\Mod\dA)/\Ho (\Ac\dA).
\]

It is also natural to consider the category of h-projective DG modules.
We call a  DG $\dA$\!--module $\mP$ {\it h-projective (homotopically projective)} if
$$\dHom_{\Ho(\Mod\dA)}(\mP, \mN)=0$$
for every acyclic DG module $\mN$ (dually, we can define {\it h-injective} DG modules).
Let  $\dP(\dA)\subset \Mod\dA$ denote the full
DG subcategory of h-projective objects. It can be easily checked that a semi-free
DG-module is h-projective and the natural embedding $\SF\dA\hookrightarrow\dP(\dA)$
is a quasi-equivalence.
Moreover, the canonical DG functors $\SF\dA\hookrightarrow\dP(\dA)\hookrightarrow\Mod\dA$ induce equivalences
$\Ho(\SF\dA)\stackrel{\sim}{\to} \Ho(\dP(\dA))\stackrel{\sim}{\to} \D(\dA)$ of triangulated categories.

Let $\mF:\dA \to \dB$ be a DG functor between small DG categories.
It induces the restriction DG functor
\[
\mF_*:\Mod\dB\lto \Mod\dA
\]
which sends a DG $\dB$\!--module $\mN$ to $\mN\circ\mF.$

The restriction functor $\mF_*$ has left and right adjoint functors $\mF^*, \mF^{!}$ that are defined as follows
\[
\mF^*M(Y)=\mM\otimes_{\dA} \mF_* \mh_Y,\quad \mF^{!}M(Y)=\dHom(\mF_*\mh_Y, \mM),\quad \text{where}\quad Y\in \dB \; \text{and}\; \mM\in\Mod\dA.
\]
The DG functor $\mF^*$ is called the induction functor and it is an extension of $\mF$ on the category of DG modules, i.e there is an isomorphism of DG functors
$\mF^* \mh^\bullet_{\dA}\cong \mh^\bullet_{\dB}\mF.$

The DG functor $\mF_*$ preserves acyclic DG modules and induces a derived functor $F_*: \D(\dB)\to \D(\dA).$
Existence of h-projective and h-injective resolutions allows us to define derived functors
$\bL F^*$ and $\bR F^!$ from $\D(\dA)$ to $\D(\dB).$

More generally, let $\mT$ be an
$\dA\hy\dB$\!--bimodule that is, by definition, a DG-module over $\dA^{op}\otimes\dB.$
For each DG $\dA$\!--module $\mM$ we obtain a DG $\dB$\!--module
$\mM\otimes_{\dA} \mT.$
The DG functor $(-)\otimes_{\dA} \mT: \Mod\dA \to \Mod\dB$ admits a right adjoint
$\dHom_{\dB} (\mT, -).$
These functors do not respect
quasi-isomorphisms in general, but they form a Quillen adjunction
and the derived functors
$(-)\stackrel{\bL}{\otimes}_{\dA}\mT$ and
$\bR \Hom_{\dB} (\mT, -)$
form an adjoint pair of functors between derived categories $\D(\dA)$ and
$\D(\dB).$

\subsection{Pretriangulated DG categories, categories of perfect DG modules, and enhancements}\label{enhancements}

For any DG category $\dA$ there exist a DG category $\dA^{\ptr}$ that is called the {\it pretriangulated hull}
and a canonical fully faithful DG
functor $\dA\hookrightarrow\dA^{\ptr}.$
The idea of the definition of $\dA^{\ptr}$ is to formally add to $\dA$
all shifts, all cones, cones of morphisms between cones and etc.
There is a canonical fully faithful DG functor
(the Yoneda embedding) $\dA^{\ptr}\to \Mod\dA,$ and under this embedding
$\dA^{\ptr}$ is DG-equivalent to the DG category of finitely generated semi-free DG modules $\SFf\dA.$
We will not make a difference between the DG categories $\dA^{\ptr}$ and $\SFf\dA.$

\begin{definition} A DG category $\dA$ is called {\em pretriangulated}  if the canonical DG functor $\dA\to\dA^{\ptr}$
is a quasi-equivalence.
\end{definition}
\begin{remark}{\rm
 It is equivalent to require that the homotopy category $\Ho(\dA)$ is triangulated
as a subcategory of $\Ho(\Mod\dA).$
}
\end{remark}
The DG category $\dA^{\ptr}$ is always  pretriangulated,
so $\Ho(\dA^{\ptr})$ is a triangulated category. We denote $\tr(\dA):=\Ho(\dA^{\ptr}).$

With any small DG category $\dA$ we can also associate another DG category   $\prfdg\dA$ that is called the DG category of perfect DG modules.
 This category is even more important than $\dA^{\ptr}.$
\begin{definition} A DG category of perfect DG modules $\prfdg\dA$
is the full DG subcategory of $\SF\dA$ consisting of all DG modules which are homotopy
equivalent to a direct summand of a finitely generated semi-free DG module.
\end{definition}

Thus, the DG category $\prfdg\dA$ is pretriangulated and contains $\SFf\dA\cong \dA^{\ptr}.$ Denote by $\prf\dA$ the homotopy
category $\Ho(\prfdg\dA).$ The triangulated category $\prf\dA$ can be obtained from the triangulated category
$\tr(\dA)$ as its idempotent completion (Karubian envelope).

\begin{proposition}
For any  small DG category $\dA$ the set of representable objects  $\{\mh^Y\}_{Y\in \dA}$ forms a set of
compact generators of $\D(\dA)$ and the subcategory of compact objects $\D(\dA)^c$ coincides with the subcategory of perfect
DG modules $\prf\dA.$
\end{proposition}

\begin{remark}\label{perfect}
{\rm
If $\dA$ is a small pretriangulated DG category and $\Ho(\dA)$ is idempotent complete, then
the natural Yoneda DG functor $\mh: \dA\to\prfdg\dA$ is a quasi-equivalence.
}
\end{remark}
It is well-known that the categories $\D(\dA)$ and $\prf\dA$ are invariant under quasi-equivalences of DG categories.

\begin{proposition}\label{pre-tr_equivalence}
If a DG functor $\mF: \dA\to\dB$ is a quasi-equivalence, then the functors
\[
\mF^{*}: \SFf\dA\lto\SFf\dB,\quad \mF^{*}:\prfdg\dA\lto \prfdg\dB, \quad \mF^*:\SF\dA\lto \SF\dB
\]
are quasi-equivalences too.
\end{proposition}

Furthermore, we have the
following proposition  that is essentially equal to
Lemma 4.2 in \cite{Ke} (see also  \cite[Prop. 1.15]{LO} and proof there).

\begin{proposition}\cite{Ke}\label{Keller2}
Let $\mF: \dA\hookrightarrow \dB$ be a full embedding of DG categories
and let $\mF^{*}:\SF\dA\to\SF\dB$ (resp. $\mF^*:\prfdg\dA\to\prfdg\dB$) be the extension DG functor. Then the induced
homotopy functor $F^*: \D(\dA)\to \D(\dB)$ (resp. $F^*: \prf\dA\to \prf\dB$) is fully faithful.
If, in addition, the category $\prf\dB$ is classically generated by $\Ob\dA,$ then
$F^*$ is an equivalence.
\end{proposition}

\begin{remark}\label{Keller2_fg}
{\rm
The first statement holds for the functor
$\mF^*:\SFf\dA\to\SFf\dB$ too. The second also holds if we ask that the category
$\tr(\dB)$ is generated by $\Ob\dA$ (not classically).
}
\end{remark}

\begin{definition} Let $\T$ be a triangulated category. An {\em
enhancement} of $\T$ is a pair $(\dA , \varepsilon),$ where $\dA$ is a
pretriangulated DG category and $\varepsilon:\Ho(\dA)\stackrel{\sim}{\to} \T$ is an exact equivalence.
\end{definition}

\subsection{Quasi-functors}

Let $\kk$ be a field. Denote by
$\DGcat_k$ the category of small DG $\kk$\!--linear categories.
It is known  that it admits a structure of cofibrantly
generated model category whose weak equivalences are the
quasi-equivalences (see \cite{Ta}).
This implies that the localization $\Hqe$ of
$\DGcat_k$ with respect to the quasi-equivalences has small
$\Hom$\!-sets. This also gives that a morphism from $\dA$ to $\dB$
in the localization can be represented as $\dA\leftarrow \dA_{cof}\to\dB,$ where $\dA\leftarrow \dA_{cof}$ is a cofibrant replacement.
 It is not easy to compute  the morphism sets in the localization
category $\Hqe$ using a cofibrant replacement. On the other hand, they can be
described in term of quasi-functors.

Consider two  small DG categories $\dA$ and $\dB.$
Let $\mT$ be a $\dA\hy\dB$\!--bimodule. It defines a derived tensor functor
\[
(-)\stackrel{\bL}{\otimes}_\dA \mT: \D(\dA) \lto \D(\dB)
\]
between derived categories of  DG modules over $\dA$ and $\dB.$
\begin{definition}
An $\dA\hy\dB$\!--bimodule $\mT$ is called a {\em quasi-functor}
from $\dA$ to $\dB$ if the tensor functor
$
(-)\stackrel{\bL}{\otimes}_\dA \mT: \D(\dA) \to \D(\dB)
$
takes every representable $\dA$\!--module to an object which is
isomorphic to a representable $\dB$\!--module.
\end{definition}
Denote by  $\Rep(\dA,\; \dB)$ the full subcategory of the derived
category $\D(\dA^{op}\otimes\dB)$ of $\dA\hy\dB$\!--bimodules consisting of
all quasi-functors.
In
other words a quasi-functor is represented by a DG functor $\dA\to \Mod\dB$ whose essential image consists of
quasi-representable DG $\dB$\!--modules (``quasi-representable'' means quasi-isomorphic to a representable DG module).
Since the category of quasi-representable DG
$\dB$\!--modules is equivalent to $\Ho(\dB)$ a quasi-functor
$\mT \in \Rep(\dA,\; \dB)$ defines a functor $\Ho(\mT):\Ho(\dA)\to \Ho(\dB).$
Notice that a quasi-functor $\mF:\dA\to\dB$
defines an exact functor $\tr(\dA)\to \tr(\dB)$ between triangulated categories.

It is now known that quasi-representable functors form morphisms between DG categories in the localization
category
$\Hqe.$
\begin{theorem}\label{quasi-functors}\cite{To}  The morphisms from $\dA$ to $\dB$
in the localization $\Hqe$ of $\DGcat_{\kk}$ with respect to
quasi-equivalences are in natural bijection with the isomorphism
classes of $\Rep(\dA,\dB).$
\end{theorem}

Due to this theorem any morphism from $\dA$ to $\dB$ in the localization category $\Hqe$
will be called a quasi-functor.

Let $\mF: \dA\to\dB$ be a quasi-functor. It can be realized as
a roof  $\dA\stackrel{\ma}{\stackrel{\sim}{\longleftarrow}}\dA'\stackrel{\mF'}{\lto}\dB,$
where $\ma$  and $\mF'$ are DG functors and $\ma$ is also a quasi-equivalence.
For instance we can take a cofibrant replacement $\dA_{cof}$ as  $\dA'.$
The quasi-functor $\mF$ induces functors
\begin{equation}\label{derived_quasi}
\bL F^*= \bL F^{'*}\circ a_*: \D(\dA)\lto \D(\dB)
\quad
\text{and}
\quad
\bR F_*:= F'_* \circ\bL a^{*}: \D(\dB)\lto \D(\dA).
\end{equation}
If now we consider the quasi-functor $F$ as an $\dA\hy\dB$\!--bimodule $\mT,$ then
there are isomorphisms of functors
\[
\bL F^*\cong -\stackrel{\bL}{\otimes}_\dA \mT: \D(\dA) \lto \D(\dB)
\quad
\text{and}
\quad
\bR F_*\cong \bR\Hom_{\dB}(\mT, -): \D(\dB) \lto \D(\dA).
\]

The standard tensor product $\otimes$ on the category $\DGcat_{\kk}$ induces a tensor product
$\stackrel{\bL}{\otimes}$ on the localization $\Hqe.$
It is proved in \cite{To} that the monoidal category $(\Hqe, \stackrel{\bL}{\otimes})$ has internal Hom-functor
$\dR\lHom.$ In particular, there is a  quasi-equivalence
\begin{equation}\label{RHom}
\dR\lHom(\dA\otimes\dB,\; \dC)\cong \dR\lHom(\dA, \; \dR\lHom(\dB,\; \dC)).
\end{equation}
\begin{theorem}\cite{To}
For any DG categories $\dA$ and $\dB$ the DG category
$\dR\lHom(\dA,  \dB)$ is quasi-equivalent to the full DG subcategory $\dRep(\dA, \dB)\subset \SF(\dA^{\op}\otimes\dB)$
consisting of all objects of $\Rep(\dA, \dB).$
\end{theorem}
Thus, there are equivalences
$
\Ho(\dR\lHom(\dA,  \dB))\cong\Ho(\dRep(\dA, \dB))\cong \Rep(\dA, \dB).
$

\section{Commutative and noncommutative  schemes}

\subsection{Derived categories of quasi-coherent sheaves and noncommutative schemes}\label{3.1}
In this paper we will consider separated noetherian schemes over an arbitrary field $\kk.$
Let $X$ be such a scheme.
The abelian category $\Qcoh X$ of quasi-coherent sheaves
$\Qcoh X$ is a Grothendieck category and has enough injectives.

Denote by $\dCom X$ the DG category of unbounded complexes of quasi-coherent sheaves on $X.$
This category has enough h-injective complexes (see, e.g. \cite{KSh}). Denote by $\dI(X)$ the full DG subcategory
of h-injective complexes. This DG category gives us a natural DG enhancement for the unbounded derived category
of quasi-coherent sheaves, because $\Ho(\dI(X))\cong \D(\Qcoh X).$ Another natural
 enhancement for $\D(\Qcoh X)$ comes from the definition of the derived category and DG version
 of Verdier localization \cite{Dr}. Consider the full DG subcategory $\dAc X\subset\dCom X$ of all acyclic complexes.
 We can take the quotient DG derived category
 $
 \dCom X/\dAc X.
 $
Of course, $\dI(X)$ and $\dCom X/\dAc X$ are naturally quasi-equivalent enhancements.

There is another enhancement of $\D(\Qcoh X)$ that is very useful when we work with pullback and tensor product functors.
It comes from h-flat complexes. Recall that an (unbounded) complex $\mP^{\cdot}$ of quasi-coherent sheaves on $X$
is called {\em h-flat} if $\Tot^{\oplus}(\mP^{\cdot}\otimes_{\cO_X} \mC^{\cdot})$ is acyclic
for any acyclic $\mC^{\cdot} \in\dAc X.$ Denote by $\dHf X\subset\dCom X$ the full DG subcategory of h-flat complexes.
It was shown in \cite[Prop. 1.1]{AJL} that there are enough h-flat complexes in $\dCom X$ for any separated quasi-compact scheme.
Hence the DG quotient category $\dHf X/\dHAc X,$ where $\dHAc X$ is the DG subcategory of acyclic h-flat complexes, is an enhancement of $\D(\Qcoh X)$ (see \cite[3.10]{KL}).

It is easy to see that for any morphism of schemes $f:X\to Y$ the pullback
$f^*,$ acting componentwise on complexes,  sends h-flat complexes to h-flat complexes and h-flat acyclic
complexes to h-flat acyclic complexes. It is also true that the tensor product of an h-flat acyclic complex with
any complex is acyclic (see \cite{Sp}).
Thus, for any morphism of schemes $f: X\to Y$ we obtain a DG functor (not only a quasi-functor)
\[
\mf^*: \dHf Y/\dHAc Y\lto \dHf X/\dHAc X,
\]
which induces the derived inverse image functor $\bL f^*$ on the derived categories of quasi-coherent sheaves.
Similarly, we have a DG tensor functor $(-) \otimes \mP^{\cdot}$  from $\dHf X/\dHAc X$ to itself.

Thus, we have three different DG categories $\dI(X),$ $\dCom X/\dAc X,$ and $\dHf X/\dHAc X,$ which are natural quasi-equivalent
enhancements for $\D(\Qcoh X).$ There is no reason to make difference between them, but sometimes one of them is more favorable
because some quasi-functors can be realized as usual DG functors.
In this paper we work with the DG category $\dHf X/\dHAc X,$ which will be denoted by $\dD(\Qcoh X),$
since pulbacks and tensor products are DG functors on them.

For any morphism of schemes $f: X\to Y$ we also have a DG functor $\mf_*$ from $\dI(X)$ to
$\dCom Y/\dAc Y$ acting  componentwise on h-injective complexes. This DG functor induces a quasi-functor that we will denote by the same letter
\[
\mf_*: \dD(\Qcoh X)\stackrel{\sim}{\lto}\dCom X/\dAc X\stackrel{\sim}{\longleftarrow}\dI(X)\stackrel{\mf_*}{\lto}
\dCom Y/\dAc Y
\stackrel{\sim}{\longleftarrow}\dD(\Qcoh X).
\]

Recall now the important notion of a perfect complex on a scheme $X,$ which was introduced in \cite{SGA6}. A
perfect complex is a complex of sheaves which is locally
quasi-isomorphic to a bounded complex of locally free sheaves of finite
type (a good reference is \cite{TT}).
\begin{definition}
Denote by $\prfdg X$ the full DG subcategory of $\dD(\Qcoh X)$ consisting of all perfect complexes.
\end{definition}

The triangulated category $\prf X=\Ho(\prfdg X)$ is a full subcategory of $\D(\Qcoh X).$
Amnon Neeman in \cite{Ne} showed that
the triangulated category $\D(\Qcoh X)$ is compactly generated and $\prf X$ is nothing but the subcategory of compact object in $\D(\Qcoh X).$
In \cite{Ne} this assertion is proved for any quasi-compact and separated scheme, in \cite{BVdB} a generalization of this fact for
$\D_{\Qcoh}(X)$ was established for a quasi-compact and quasi-separated scheme.

For any morphism of schemes $f: X\to Y$ the DG functor $\mf^*$ induces a DG functor
\[
\mf^*:\prfdg Y\to \prfdg X.
\]

Under some conditions on the morphism $f:X\to Y,$ the induced quasi-functor $\mf_*$ from
$\dD(\Qcoh X)$ to $\dD(\Qcoh Y)$ sends $\prfdg X$ to $\prfdg Y$ (see \cite[2.5.4]{TT} and \cite[III]{SGA6}).
Thus, for noetherian schemes $X$ and $Y$ if $f$ is proper and has finite
Tor-dimension we obtain a quasi-functor
\[
\mf_*: \prfdg X\lto \prfdg Y.
\]
Note that it holds for any morphism between smooth and proper schemes.

In \cite{Ne, BVdB} it was proved that the category $\prf X$ admits a classical generator $\mE$ and, hence, $\mE$
is a compact generator of the whole $\D(\Qcoh X).$
Let us take such a generator $\mE\in\prfdg X.$ Denote by $\dE$ its DG algebra of endomorphisms,
i.e. $\dE=\dHom(\mE, \mE).$ Since $\mE$ is perfect, the DG algebra $\dE$ has only finitely many cohomologies.
Proposition \ref{Keller2} implies the following statement.
\begin{statement}\cite[3.1.8]{BVdB}\label{statement}
The DG category $\dD(\Qcoh X)$ is quasi-equivalent to $\SF\dE$ and $\prfdg X$ is quasi-equivalent to
$\prfdg \dE,$ where $\dE$ is a DG algebra with bounded cohomology.
\end{statement}

This fact allows us to suggest a definition of a (derived) noncommutative scheme over $\kk.$
\begin{definition}\label{noncommutative_scheme}
A {\em (derived) noncommutative scheme} over a field $\kk$ is a $\kk$\!--linear DG category of the form $\prfdg\dE,$
where $\dE$ is a cohomologically  bounded DG algebra over $\kk.$ The derived category
$\D(\E)$ will be called the derived category of quasi-coherent sheaves on this noncommutative scheme.
\end{definition}

For a noetherian scheme $X$ we consider the abelian category of coherent sheaves
$\coh X.$
Denote by $\D^b(\coh(X))$
the bounded derived category of coherent
sheaves on $X.$
Since $X$ is noetherian the natural functor $\D^b(\coh(X))\to
\D(\Qcoh(X))$ is fully faithful and realizes an equivalence of
$\D^b(\coh(X))$ with the full subcategory
$\D^b(\Qcoh(X))_{\coh}\subset \D(\Qcoh(X))$ consisting of all cohomologically bounded
complexes with coherent cohomology (see \cite[II 2.2.2]{SGA6}).
Because of that,  when we  consider $\D^b(\coh(X))$ as a subcategory of
$\D(\Qcoh(X))$ we will identify it with the full subcategory
$\D^b(\Qcoh(X))_{\coh},$ adding all isomorphic objects.
The enhancement  $\dD(\Qcoh X)$ induces an enhancement of $\D^b(\coh X)$ that we denote
by $\dD^b(\coh X).$

\subsection{Gluing of DG categories}

Let $\dA$ and $\dB$ be two small DG categories and let $\mS$ be a $\dB\hy\dA$\!--bimodule,
i.e. a DG  $\dB^{\op}\otimes\dA$\!--module.
We now construct a so called upper triangular DG category corresponding to  the data
$(\dA, \dB, \mS).$

\begin{definition}\label{upper_tr}
Let $\dA$ and $\dB$ be two small DG categories and let $\mS$ be a $\dB\hy\dA$\!--bimodule.
The {\em upper triangular} DG category $\dC=\dA\underset{\mS}{\with}\dB$ is defined as follows:
\begin{enumerate}
\item[1)] $\Ob(\dC)=\Ob(\dA)\bigsqcup\Ob(\dB),$

\item[2)]
$
\dHom_{\dC}(X, Y)=
\begin{cases}
 \dHom_{\dA}(X, Y), & \text{ when $X,Y\in\dA$}\\
     \dHom_{\dB}(X, Y), & \text{ when $X,Y\in\dB$}\\
      \mS(Y, X), & \text{ when $X\in\dA, Y\in\dB$}\\
      0, & \text{ when $X\in\dB, Y\in\dA$}
\end{cases}
$
\end{enumerate}
with evident composition law coming from DG categories $\dA, \dB$ and the bimodule structure on $\mS.$
\end{definition}

The upper triangular DG category $\dC=\dA\underset{\mS}{\with}\dB$ is not necessary pretriangulated even if the components $\dA$ and $\dB$
are pretriangulated. To make this operation well defined on the class of pretriangulated categories we introduce a so called gluing
of pretriangulated categories.

\begin{definition}\label{gluing_cat}
Let $\dA$ and $\dB$ be two small pretriangulated DG categories and let $\mS$ be a $\dB\hy\dA$\!--bimodule.
 A {\em gluing} $\dA\underset{\mS}{\oright}\dB$ of DG categories $\dA$ and $\dB$ via $\mS$ is defined as the pretriangulated hull of $\dA\underset{\mS}{\with}\dB,$ i.e. $\dA\underset{\mS}{\oright}\dB=(\dA\underset{\mS}{\with}\dB)^{\ptr}.$
\end{definition}

\begin{remark}{\rm
The gluing can be defined for any DG categories not only for pretriangulated
(see, for example, \cite{KL}). The resulting DG category is not necessary pretriangulated.
However, we prefer to restrict ourself to the pretriangulated case,
because the definition above is more convenient
for our purposes. Since we use different definition we give different proofs for
Propositions \ref{dg_semiorhtogonal}, \ref{gluing_semi-orthogonal}, and \ref{gluing_quasifunctors} in spite of they were also proved
in \cite{KL}.
}
\end{remark}

Natural fully faithful DG inclusions $\ma: \dA\hookrightarrow \dA\underset{\mS}{\with}\dB$ and
$\mb: \dB\hookrightarrow \dA\underset{\mS}{\with}\dB$ induce  fully faithful DG
functors $\ma^*: \dA\hookrightarrow \dA\underset{\mS}{\oright}\dB$ and
$\mb^*: \dB\hookrightarrow \dA\underset{\mS}{\oright}\dB.$

It is easy to see that the restriction functor $\mb_*: \Mod(\dA\underset{\mS}{\with}\dB) \to \Mod\dB$
sends semi-free DG modules to semi-free and we obtain a DG functor
$\SFf(\dA\underset{\mS}{\with}\dB)\to\SFf\dB.$ By assumption $\dB$ is pretriangulated, and
we know that the pretriangulated hull is DG-equivalent to the DG category of finitely generated semi-free DG modules.
Thus we obtain a quasi-functor $\mb_*: \dA\underset{\mS}{\oright}\dB \to \dB$ that is right adjoint to
$\mb^*.$
These quasi-functors induce exact functors
\[
a^*:\Ho(\dA)\lto\Ho(\dA\underset{\mS}{\oright}\dB), \quad b^*:\Ho(\dB)\lto\Ho(\dA\underset{\mS}{\oright}\dB),\quad b_*: \Ho(\dA\underset{\mS}{\oright}\dB)\lto \Ho(\dB)
\]
between triangulate categories such that $a^*, b^*$ are fully faithful, and $b_*$ is right adjoint to $b^*.$
Therefore, there is a semi-orthogonal decomposition
\[
\Ho(\dA\underset{\mS}{\oright}\dB)=\langle \N,\; \Ho(\dB)\rangle
\]
with some triangulated subcategory $\N.$ It is evident that $\Ho(\dA)$ is a full subcategory of $\N$ with respect to the functor $a^*.$
The subcategory $\N$ is left admissible and we have the quotient functor $\Ho(\dA\underset{\mS}{\oright}\dB)\to\N$ that sends
$\Ho(\dB)$ to zero. Since the category $\Ho(\dA\underset{\mS}{\oright}\dB)$ is generated by the union of objects
$a^* \mh^X_{\dA}$ and $b^*\mh^Y_{\dB}$ we obtain that the subcategory $\N$ is generated by the objects $a^* \mh^X_{\dA}.$
Hence $\N$ coincides with the triangulated subcategory $\Ho(\dA)\subset\N,$ because it also contains all these objects.
Thus, we have proved the following proposition.

\begin{proposition}\label{dg_semiorhtogonal} Let the DG category $\dC$ be the gluing $ \dA\underset{\mS}{\oright}\dB.$
Then the DG functors $\ma^*: \dA\to\dC$ and $\mb^*: \dB\to\dC$ induce a semi-orthogonal decomposition
for the triangulated category $\Ho(\dC)$ of the form
$\Ho(\dC)=\langle\Ho(\dA), \Ho(\dB)\rangle.$
\end{proposition}

On the other hand, we can show that any enhancement of a  triangulated category with a semi-orthogonal decomposition
can be obtained as a gluing of enhancements of the summands.

\begin{proposition}\label{gluing_semi-orthogonal}
Let $\dC$ be a pretriangulated DG category. Suppose that we have a semi-orthogonal decomposition
$\Ho(\dC)=\langle \A, \B\rangle.$ Then the DG category $\dC$ is quasi-equivalent to the gluing
$\dA\underset{\mS}{\oright}\dB,$ where $\dA, \dB\subset\dC$ are full DG subcategories
with the same objects as $\A$ and $\B,$ respectively, and the $\dB\hy\dA$\!--bimodule is given by the rule
\begin{equation}\label{bimodule}
\mS ( Y, X)=\dHom_{\dC}(X, Y), \quad \text{with}\quad X\in\dA \;\text{and}\; Y\in\dB.
\end{equation}
\end{proposition}
\begin{dok}
Take full DG subcategories $\dA\subset\dC$ and $\dB\subset\dC$ with objects from $\A$ and $\B,$ respectively.
Consider the $\dB\hy\dA$\!--bimodule $\mS$ defined by rule (\ref{bimodule}).
There is a natural inclusion of the upper triangular DG category $\dA\underset{\mS}{\with}\dB$ into $\dC.$
Since $\dC$ is pretriangulated we obtain a quasi-functor from the pretriangulated hull
$\dA\underset{\mS}{\oright}\dB$ to $\dC.$

Since $\A$ and $\B$ are semi-orthogonal, the DG category $\dA\underset{\mS}{\with}\dB$ under the inclusion $\dA\underset{\mS}{\with}\dB \hookrightarrow \dC$ is  quasi-equivalent to the full
DG subcategory of $\dC$ on the set of objects $\Ob(\dA)\bigsqcup\Ob(\dB).$
Combining  Propositions \ref{pre-tr_equivalence}, \ref{Keller2}, and Remark \ref{Keller2_fg} we obtain that the functor
$\Ho(\dA\underset{\mS}{\oright}\dB)\to\Ho(\dC)$ is fully faithful.
Since the set $\Ob(\dA)\bigsqcup\Ob(\dB)$ generates the category $\Ho(\dC),$  this functor is an equivalence
by Remark \ref{Keller2_fg}.
\end{dok}

\begin{example}\label{ex1}
{\rm Let $X$ be a noetherian scheme and let $\E$ be a vector bundle on $X$ of rank $2.$
Consider the projectivization $\PP(\E^{\vee})$ with projection $p: \PP(\E^{\vee})\to X.$
Denote by $\cO(1)$ the antitautological line bundle on $\PP(\E^{\vee}).$ We know that $\bR p_*\cO(1)\cong \E$
and $p^*$ is fully faithful.
It was shown in \cite{Blow} that there is a semi-orthogonal decomposition of the form
\[
\prf \PP(\E^{\vee})=\langle p^* \prf X,\; p^* \prf X \otimes \cO(1)\rangle.
\]
 Furthermore,
the DG category $\prfdg \PP(\E^{\vee})$ is quasi-equivalent to the gluing
$\prfdg X \underset{\mS_{\E}}{\oright}\prfdg X,$ where $\mS_{\E}$ is a DG bimodule of the form
\[
\mS_{\E}(B, A)\cong\dHom_{\prfdg X}(A,\; B\otimes\E),\quad\text{where}\quad A, B\in\prfdg X.
\]
By the same rule the DG category $\dD^b(\coh \PP(\E^{\vee}))$
can be obtain as the  gluing of $\dD^b(\coh X)$ with itself via $\mS_{\E}.$}
\end{example}

\begin{example}\label{ex2}
{\rm
Let $\pi: \wt{X}\to X$ be a blowup of a regular scheme $X$ along a closed regular subscheme $Y$ of codimension $2.$
The functor $\bL\pi^*$ is fully faithful. Consider the exceptional divisor $j: E\hookrightarrow \wt{X}.$
The morphism $p: E\to Y$ is the projectivization of the normal bundle to $Y$ in $X.$
The functor $\bR j_* p^*$ is fully faithful as well.
The triangulated category $\prf \wt{X}$ has a semi-orthogonal decomposition of the form
\[
\prf \wt{X}=\langle \bL\pi^* \prf X,\; \bR j_* p^*\prf Y\rangle.
\]
Furthermore,
the DG category $\prfdg \wt{X}$ is quasi-equivalent to the gluing
$\prfdg X \underset{\mS}{\oright}\prfdg Y,$ where $\mS$ is a DG bimodule of the form
\[
\mS(B, A)\cong\dHom_{\prfdg Y}(\mi^* A,\;  B),\quad\text{where}\quad A\in\prfdg X,\; B\in\prfdg Y,\quad\text{and}\quad i: Y\hookrightarrow X.
\]
By the same rule the DG category $\dD^b(\coh \wt{X})$
can be obtained as  gluing via $\mS.$

}
\end{example}

Let $\ma:\dA\to \dA'$ and $\mb:\dB\to \dB'$ be DG functors between small pretriangulated  DG categories. Let $\mS$ and $\mS'$ be bimodules, i.e.
DG modules over $\dB^{\op}\otimes\dA$ and ${\dB'}^{\op}\otimes\dA'$ respectively. Consider the restriction functor on bimodules
\[
(\mb\otimes\ma)_*: \Mod ({\dB'}^{\op}\otimes\dA')\to \Mod (\dB^{\op}\otimes\dA).
\]
Suppose that we have a map of DG modules $\phi:\mS\to (\mb\otimes\ma)_*\mS'.$ Then it is evident from Definition \ref{upper_tr}
that there are DG functors
\[
\ma\underset{\phi}{\with}\mb: \dA\underset{\mS}{\with}\dB \lto \dA'\underset{\mS'}{\with}\dB',
\quad
\text{ and }\quad
\ma\underset{\phi}{\oright}\mb: \dA\underset{\mS}{\oright}\dB \lto \dA'\underset{\mS'}{\oright}\dB'.
\]

Furthermore, assume that $\phi$ is a quasi-isomorphism.
Now if the exact functors $a:\Ho(\dA)\to\Ho(\dA')$ and $b:\Ho(\dB)\to \Ho(\dB')$ are fully faithful, then
\[
a\underset{\phi}{\with}b: \Ho(\dA\underset{\mS}{\with}\dB) \lto \Ho(\dA'\underset{\mS'}{\with}\dB'),
\quad
\text{and}\quad
a\underset{\phi}{\oright}b: \Ho(\dA\underset{\mS}{\oright}\dB) \lto \Ho(\dA'\underset{\mS'}{\oright}\dB')
\]
are fully faithful by Theorem \ref{Keller2} and Remark \ref{Keller2_fg}.

If $\ma$ and $\mb$ are quasi-equivalences and $\phi$ is a quasi-isomorphism, then
$\ma\underset{\phi}{\with}\mb$ is a quasi-equivalence and, by Remark \ref{Keller2_fg},  $\ma\underset{\phi}{\oright}\mb$
is a quasi-equivalence too since the objects of $\dA\underset{\mS}{\with}\dB$ generate $\Ho(\dA'\underset{\mS'}{\oright}\dB')$
in this case.

This statement can be generalized to a class of quasi-functors.
Indeed, quasi-functors  $\ma:\dA\to \dA'$ and $\mb:\dB\to \dB'$ induce a quasi-functor
$\mb\otimes\ma: \dB^{\op}\otimes\dA \to {\dB'}^{\op}\otimes\dA'.$ The quasi-functor $\mb\otimes\ma$
induces a derived functor
\[
\bR (\mb\otimes\ma)_*: \D({\dB'}^{\op}\otimes\dA')\lto \D(\dB^{\op}\otimes\dA)
\]
by rule (\ref{derived_quasi}).
Any quasi-functor $\ma:\dC\to \dD$ can be realized as a roof $\dC\stackrel{\sim}{\leftarrow} \dC'\to \dD$
and any morphism of bimodules $\mM\to\mN$ in $\D(\C)$ can be represented as
a roof of  the form $\mM\stackrel{\sim}{\leftarrow}\mM'\to \mN.$
Hence, we obtain the following proposition.

\begin{proposition}\label{gluing_quasifunctors}
Let $\ma:\dA\to \dA'$ and $\mb:\dB\to \dB'$ be  quasi-functors between small DG categories. Let $\mS$ and $\mS'$ be
DG modules over $\dB^{\op}\otimes\dA$ and ${\dB'}^{\op}\otimes\dA'$ respectively. Assume that there is a morphism
$\phi:\mS\to \bR (\mb\otimes\ma)_*\mS'$ in $\D(\dB^{\op}\otimes\dA).$ Then there are quasi-functors
\[
\ma\underset{\phi}{\with}\mb: \dA\underset{\mS}{\with}\dB \lto \dA'\underset{\mS'}{\with}\dB',
\quad
\text{and}\quad
\ma\underset{\phi}{\oright}\mb:\dA\underset{\mS}{\oright}\dB \lto \dA'\underset{\mS'}{\oright}\dB'.
\]
Moreover, suppose that $\phi$ is a quasi-isomorphism.
If $a:\Ho(\dA)\to\Ho(\dA')$ and $b:\Ho(\dB)\to \Ho(\dB')$ are fully faithful, then
\[
a\underset{\phi}{\with}b: \Ho(\dA\underset{\mS}{\with}\dB) \lto \Ho(\dA'\underset{\mS'}{\with}\dB'),
\quad
\text{and}\quad
a\underset{\phi}{\oright}b:\Ho(\dA\underset{\mS}{\oright}\dB) \lto \Ho(\dA'\underset{\mS'}{\oright}\dB')
\]
are fully faithful.
If $\ma, \mb$ are quasi-equivalences, then both
$\ma\underset{\phi}{\with}\mb$ and $\ma\underset{\phi}{\oright}\mb$ are quasi-equivalences.
\end{proposition}

\subsection{Regular, smooth, and proper noncommutative schemes}\label{reg_proper_noncommutative}

Let $\T$ be a small $\kk$\!--linear triangulated category and let $\dA$ be a small $\kk$\!--linear DG category.
\begin{definition}\label{reg_and_prop} We say that $\T$ is {\em regular} if it has a strong generator,
and we say that $\T$ is proper if
$\bigoplus_{m\in\ZZ}\Hom(X, Y[m])$ is finite dimensional for any two objects $X, Y\in\T.$
\end{definition}

\begin{definition} We call $\dA$ {\em regular} (resp. {\em proper}) if the triangulated category
$\prf\dA$ is regular (resp. proper).
\end{definition}
\begin{remark}
{\rm
Instead of $\prf\dA$ we can consider $\tr(\dA).$ Since $\prf\dA$ is the idempotent completion of $\tr(\dA),$
 regularity and properness of these categories hold simultaneously.
}
\end{remark}
\begin{remark}\label{finitness}
{\rm
It is easy to see that $\dA$ is proper if and only if $\bigoplus_i H^i(\dHom(X, Y))$ are finite dimensional
for all $X, Y\in\dA.$ It is evidently necessary due to Yoneda embedding $\dA\subset\prfdg\dA.$
Since $\Ob\dA$ classically generate $\prf\dA$ it is also sufficient.
}
\end{remark}

The following theorem is due to A.~Bondal and M.~Van den Bergh.
\begin{theorem}{\rm \cite[Th. 1.3]{BVdB}}\label{saturated}
Let $\T$ be a regular and proper triangulated category which is idempotent complete (Karoubian).
Then any exact functor from $\T^{\op}$ to the bounded derived category of finite dimensional
vector spaces $\prf\kk$ is representable, i.e. it is of the form $h^{Y}=\Hom(-, Y).$
\end{theorem}
Such a triangulated category is called {\em right saturated} in \cite{BK, BVdB}.
It is proved in \cite[2.6]{BK} that if $\T$ is a right saturated triangulated category and it is a full subcategory
in  a proper triangulated category, then it is right admissible there.
By Theorem \ref{saturated} a regular and proper
idempotent complete triangulated category is right saturated. Since the opposite category is also regular and proper,
it is left saturated as well. Thus, we obtain the following proposition.
\begin{proposition}\label{admissible}
Let $\T\subset\T'$ be a full subcategory in  a proper triangulated category
$\T'.$ Assume that $\T$ is regular and  idempotent complete.  Then $\T$ is admissible in $\T'.$
\end{proposition}

The proof of  Theorem \ref{saturated} works for DG categories without any changes (see \cite{BVdB}).
Moreover, the DG version can be deduced from Theorem \ref{saturated}.
\begin{theorem}\label{dg_saturated} Let $\dA$ be a small DG category that is regular and proper.
Then a DG module $\mM$ is perfect if and only if $\dim\bigoplus_i  H^i(\mM(X))<\infty$
for all $X\in\dA.$
\end{theorem}
\begin{dok}
If $\mM$ is perfect, then  $\dim\bigoplus_i  H^i(\mM(X))<\infty,$ because $\prf\dA$ is proper.

Assume now that $\dim\bigoplus_i  H^i(\mM(X))<\infty.$ This implies that $\dim\bigoplus_i  H^i(\dHom(\mP, \mM))<\infty$
for any $\mP\in \prfdg\dA.$
Therefore, the module $\mM$ gives the DG functor
$\dHom(-, \mM)$ from  $\prfdg\dA$ to $\prfdg\kk.$ By Theorem \ref{saturated} the induced functor
$\Hom(-, \mM): \prf\dA\to\prf\kk$ is represented by an object $\mN\in\prf\dA$ and there is a canonical map
$\mN\to\mM.$ The cone $C$ of this map in $\D(\dA)$ is an object such that $\Hom(X, C)=0$ for any $X\in\dA.$
This implies that $C=0$ because $\Ob\dA$ is a set of compact generators in $\D(\dA).$
Thus, $\mM$ is a perfect complex.
\end{dok}
\begin{corollary}\label{representable}
Let $\dA$ be a regular and proper pretriangulated DG category for which $\Ho(\dA)$ is idempotent complete. Let $\mM$ be a DG $\dA$\!--module such that
$\dim\oplus_i  H^i(\mM(X))<\infty$
for all $X\in\dA.$
Then $\mM$ is quasi-isomorphic to a representable module $\mh^Y=\dHom(-, Y)$ for some $Y\in\dA.$
\end{corollary}
\begin{proof}
It directly follows from the previous theorem and Remark \ref{perfect}.
\end{proof}

The properties of regularity and properness behave well under taking semi-orthogonal summands and gluing.
\begin{proposition}\label{reg_prop}
Let $\T$ be a $\kk$\!--linear triangulated category with a semi-orthogonal decomposition
$\T= \langle\T_1, \T_{2}\rangle.$ The following properties hold
\begin{enumerate}
\item[1)] if $\T$ is proper, then $\T_i$ are proper;
\item[2)] if $\T$ is regular, then $\T_i$ are regular;
\item[3)] if $\T_i,\; i=1,2$ are regular, then $\T$ is regular too.
\end{enumerate}
\end{proposition}
\begin{dok}
1) is evident, since any subcategory of a proper category is proper.
To prove 2) we should note that there are quotient functors from $\T$ to $\T_i.$ Now it is evident that the images of a strong generator
under these functors are strong generators in $\T_i.$

Let $E_i$ be strong generators of $\T_i$ such that $\langle E_i\rangle_{n_i}=\T_i.$ We can take $E=E_1\oplus E_2.$
There are embeddings $\langle E\rangle_{n_i}\supset \langle E_i\rangle_{n_i}=\T_i, i=1,2.$ By definition of
a semi-orthogonal decomposition, for any object $X\in\T$ there is an exact triangle of the form $X_2\to X\to X_1$
with $X_i\in\T_i.$ This implies that $X\in\langle E\rangle_{n_2}\diamond\langle E\rangle_{n_1}=\langle E\rangle_{n_1+n_2}.$
Hence $\T=\langle E\rangle_{n_1+n_2}.$ This proves 3).
\end{dok}
\begin{remark}
{\rm
The proof implies inequality $\dim\T\le \dim \T_1+\dim\T_2 +1.$
}
\end{remark}

\begin{proposition}
Let $\dA$ and $\dB$ be two small pretriangulated DG categories and let $\mS$ be a $\dB\hy\dA$\!--bimodule.
Then the following conditions are equivalent:
\begin{enumerate}
\item the  gluing $\dA\underset{\mS}{\oright}\dB$ is regular and proper,
\item $\dA$ and $\dB$ are regular and proper and  $\dim\bigoplus_i  H^i(\mS(Y, X))<\infty$
for all $X\in\dA, Y\in\dB.$
\end{enumerate}
\end{proposition}
\begin{dok}
(1)$\Rightarrow$(2). Since $\Ho(\dA\underset{\mS}{\oright}\dB)=\langle \Ho(\dA), \Ho(\dB)\rangle$
regularity and properness of $\dA$ and $\dB$ directly follow from Proposition \ref{reg_prop} 1) and 2).
Properness of $\dA\underset{\mS}{\oright}\dB$ implies that $\dim\bigoplus_i  H^i(\mS(Y, X))<\infty$ as well.

(2)$\Rightarrow$(1). Regularity of the gluing follows from 3) of Proposition \ref{reg_prop} and Proposition \ref{dg_semiorhtogonal}.
In view of Remark \ref{finitness} properness of $\dA\underset{\mS}{\oright}\dB$ directly follows from the properness of $\dA$ and $\dB$ and the finiteness of $\mS.$
\end{dok}

There is another important property of DG categories that is called smoothness.
\begin{definition}
A small $\kk$\!--linear DG category $\dA$ is called $\kk$\!-smooth if it is perfect as the module over  $\dA^{\op}\otimes\dA.$
\end{definition}
This property depends on the base field $\kk.$ For example, a finite inseparable
extension $F\supset\kk$ is not smooth over $\kk$ and it is smooth over itself.

The following statement is proved in \cite{Lu} see Lemmas 3.5. and 3.6.
\begin{proposition}\label{smooth_regular}
If a small DG category $\dA$ is smooth, then it is regular.
\end{proposition}

Smoothness is invariant under Morita equivalence \cite{Lu, LS}. This means that if $\D(\dA)$ and $\D(\dB)$
are equivalent through a functor of the form $(-)\stackrel{\bL}{\otimes_{\dA}}\mT,$ where
$\mT$ is an $\dA\hy\dB$\!--bimodule, then $\dA$ is smooth if and only if $\dB$ is smooth.

Since $\dA\underset{\mS}{\with}\dB$ and $\dA\underset{\mS}{\oright}\dB$ are Morita equivalent we obtain that
smoothness of $\dA\underset{\mS}{\with}\dB$ and $\dA\underset{\mS}{\oright}\dB$ hold simultaneously.
Further, we can compare smoothness of a gluing with smoothness of summands. We get the following.
\begin{theorem}\cite[3.24]{LS}\label{smooth_glue}
Let $\dA$ and $\dB$ be two small pretriangulated DG categories over a field $\kk$ and let $\mS$ be a $\dB^{\op}\otimes\dA$\!--module.
Then the following conditions are equivalent:
\begin{enumerate}
\item the gluing $\dA\underset{\mS}{\oright}\dB$ is smooth;
\item $\dA$ and $\dB$ are smooth and $\mS$ is a perfect $\dB^{\op}\otimes\dA$\!--module.
\end{enumerate}
\end{theorem}

\subsection{Regularity, smoothness, and properness in commutative geometry}
Let us now discuss all these properties of DG categories in context of the usual geometry of schemes.

\begin{proposition}\label{proper}
Let $X$ be a proper scheme. Then the category $\prf\!X$ is proper.
\end{proposition}
\begin{dok}
Let $\E^{\cdot}$ be a perfect complex. Consider the functor $\bR\lHom(\E^{\cdot}, -)$ from $\D(\Qcoh X)$ to itself.
Since a perfect complex is locally quasi-isomorphic to a finite complex of vector bundles we obtain that
any object $\bR\lHom(\E^{\cdot}, \F^{\cdot})$ is perfect when $\E^{\cdot}$ and $\F^{\cdot}$ are perfect.
Let $\pi$ be the canonical morphism from $X$ to $\Spec\kk.$
By \cite[III 4.8.1]{SGA6} (see also  \cite[2.5.4]{TT}) since $X$ is proper the object $\bR \pi_* \E^{\cdot}$ is perfect over $\kk$
when $\E^{\cdot}$ is perfect. Hence, the complex
\[
\bR\Hom(\E^{\cdot}, \F^{\cdot})\cong \bR\pi_*\bR\lHom(\E^{\cdot}, \F^{\cdot})
\]
is a perfect complex of $\kk$\!-vector spaces, i.e. $\bigoplus_{k}\Hom(\E^{\cdot}, \F^{\cdot}[k])$ is finite dimensional.
\end{dok}

\begin{theorem}\label{regular}
Let $X$ be a separated noetherian scheme of finite Krull dimension  over an arbitrary filed $\kk.$
Assume that the square $X\times X$ is noetherian too. Then  the following conditions are equivalent:
\begin{enumerate}
\item $X$ is regular;
\item $\prf X$ is regular, i.e. it has a strong generator.
\end{enumerate}
\end{theorem}
\begin{dok}
At first, note that the affine case $X=\Spec A$ was treated in Corollary 8.4 of \cite{Chr} and the remark immediately following (see also \cite[7.25]{Ro}).
We will use it.

(2)$\Rightarrow$(1)
Take an affine open subset $U\subset X.$ Any perfect complex on $U$ is a direct summand of a perfect
complex restricted from $X$ (\cite[Lemma 2.6]{Ne}). Hence, the category $\prf U$ is strongly generated too.
Thus we have reduced to the affine case.
If $\prf A$ is strongly generated, then the algebra $A$ has  finite global dimension, i.e. it is regular.

(1)$\Rightarrow$(2) By \cite[II 2.2.7.1]{SGA6} any regular separated noetherian scheme has an ample family of line bundles, i.e.
there is a family of line bundles $\{\L_{\alpha}\}$ on $X$  such that for any quasi-coherent sheaf $\F,$ the evaluation map
\[
\bigoplus_{\alpha;\; n\ge 1}\Gamma(X, \F\otimes\L_{\alpha}^{\otimes n})\otimes \L_{\alpha}^{\otimes -n}\twoheadrightarrow
\F
\]
is an epimorphism.
In particular, for any coherent sheaf $\F$ there are an algebraic vector bundle $\E$ (i.e. locally free sheaf of finite type)
and an epimorphism $\E\twoheadrightarrow \F.$

Consider an affine covering $X=\bigcup_{i=1}^{m} V_i,$ where $V_i=\Spec A_i.$
Since $X$ is regular, all $A_i$ are regular noetherian algebras of finite dimension and, hence, they have finite global dimension.
This implies that for sufficiently large $n\in \ZZ$ (greater than maximum of global dimensions of $A_i$) for any quasi-coherent sheaf $\F$ there is a global locally free resolution
\begin{equation}\label{locally_free}
0\lto \E^{-n}\lto\cdots\lto\E^0\lto\F\lto 0.
\end{equation}
By \cite[1.4.12]{EGA3} (see also \cite[App. B]{TT}) there exists an integer $k\in \ZZ$ such that for all $p\ge k$
and for all quasi-coherent sheaves $\G,$ one has $\Ext^p(\E, \G)=H^p(X, \E^{\vee}\otimes \G)=0,$ where $\E$ is  locally free.
Using a locally free resolution of type (\ref{locally_free}) for a quasi-coherent sheaf $\F,$ one has that
for sufficiently large $N\in \ZZ$ for all $p\ge N$ and all quasi-coherent sheaves $\F, \G,$ we have
$\Ext^p(\F, \G)=0.$ Thus, the abelian category $\Qcoh(X)$ has a finite global dimension. Let us denote it by
$k=\mathrm{gl}.\dim\Qcoh(X).$

Consider the product $X\times_{\kk} X.$ It is known that the family $\{\L_{\alpha}^r\boxtimes\L_{\beta}^s| r,s\ge 1\}$
forms an ample family on $X\times X$ (see \cite[2.1.2.f]{TT}),  the scheme $X\times X$ not necessary being regular.

Take the structure sheaf $\cO_{\Delta}$ of the diagonal $\Delta\subset X\times X.$ Since $X$ is separated, $\Delta$ is closed. As $X\times X$ is noetherian, $\cO_{\Delta} $ is  a coherent sheaf. Fix an infinite locally free resolution $\E^{\cdot}$ of $\cO_{\Delta}$
\[
\cdots \lto \E^{-n}\lto\cdots\lto\E^0\lto\cO_{\Delta}\lto 0,
\]
where each $\E^{-i}$ is a finite direct sum of sheaves of the form $\L_{\alpha}^{\otimes r}\boxtimes\L_{\beta}^{\otimes s}.$
Take a brutal truncation $\sigma^{\ge -l}\E^{\cdot}$ for a sufficiently large $l\gg 0.$ It has only two cohomology sheaves
$H^{-l}(\sigma^{\ge -l}\E^{\cdot})$ and $H^0(\sigma^{\ge -l}\E^{\cdot})=\cO_{\Delta}.$
Take all $\L_{\alpha}^{\otimes n}$ that appear in $\E^{-i}$ for all $0\le i\le l$ and consider their direct sum. Denote it by
$\cS.$ We have that $\cS$ is an algebraic vector bundle and $\sigma^{\ge -l}\E^{\cdot}\in \langle \cS\boxtimes \cS\rangle_{l+1}.$

For any quasi-coherent sheaf $\F,$ the object $C=\bR\pr_{2*}(\pr_1^*(\F)\otimes \sigma^{\ge -l}\E^{\cdot})$ is a complex
on $X,$ all cohomology $H^j(C)$ of which are trivial when $j> -l+k$ except $H^0(C)$ that is isomorphic to $\F.$
Since $l$ is large enough, we obtain that $\F$ is a direct summand of $C.$ But $C$ belongs to $\widebar{\langle\cS\rangle}_{l+1}.$
Therefore, $\F\in  \widebar{\langle\cS\rangle}_{l+1}$ too.
Thus, we obtain that $\widebar{\langle\cS\rangle}_{l+1}$ contains all quasi-coherent sheaves.
Now we can apply the following proposition from \cite{Ro}.

\begin{proposition}\cite[Prop. 7.22]{Ro}
Let $\A$ be an abelian category of finite global dimension $k.$
Let $C$ be a complex of objects from $\A.$
 Then, there is a distinguished
triangle in $\D(\A)$
\[
\bigoplus_{i} D_i\lto C\lto \bigoplus_{i} E_i,
\]
where $D_i=\sigma^{\ge ki+1}(\tau^{\le k(i+1)-1} C)$ is a complex with zero terms outside
$[ki+1,\dots k(i+1)-1]$ and $E_i$ is a complex concentrated in degree $ki.$
\end{proposition}

Using this proposition we obtain that any object of $\D(\Qcoh X)$ belongs to $\widebar{\langle\cS\rangle}_{k(l+1)}$
where $k=\mathrm{gl}.\dim\Qcoh(X).$ Indeed $E_i\in \widebar{\langle\cS\rangle}_{l+1}$ and $D_i$ as complexes
of length $k-1$ belong to $\widebar{\langle\cS\rangle}_{(k-1)(l+1)}.$

Finally, by Proposition \ref{compact_gen} we have that $\prf X=\D(\Qcoh(X))^c=\langle\cS\rangle_{k(l+1)}.$
\end{dok}

\begin{remark}
{\rm Recently, Amnon Neeman obtained a more general result in this direction.
In particular, the property to be noetherian for the square is not needed.
}
\end{remark}

\begin{proposition}\label{prop_scheme} Let $X$ be a separated scheme of finite type
over a field $\kk.$ Then $X$ is proper if and only if the category
of perfect complexes $\prf X$ is proper.
\end{proposition}
\begin{proof}
If $X$ is proper, then by Proposition \ref{proper}
the category $\prf X$ is proper.

Suppose that $\prf X$ is proper.
Let us show that $X$ is proper. We will prove by contradiction. Assume that $X$ is not proper.
By Chow's Lemma for any separated scheme of finite type $X$ there is a quasi-projective $X'$ with a proper map
$f: X'\to X.$ If $X$ is not proper $X'$ is not projective. Consider its closure $\widebar{X}'\in \PP^N.$
Take the complement $Y$ to $X'$ in $\widebar{X}'$ and choose a closed point $p\in Y.$
There is an irreducible and reduced projective curve $C\subset \widebar{X}'$ that contains the point $p$ and is not contained in $Y.$
Denote by $C_0\subset C$ the intersection of $C$ with $X'$ and by $\wt{C}$ and $\wt{C}_0$ the normalizations of
$C$ and $C_0$ respectively. Since $p\not\in C_0$ the complement $D$ to $\wt{C}_0$ in $\wt{C}$ is not empty.
The curve $C$ is regular, hence  $D$ is a Cartier divisor on $C.$ Since $D$ is effective it is ample (see, e.g. \cite[7.5.5]{Liu}).
This implies that $\wt{C}_0$ is an affine curve.
Now consider the composition map $g:\wt{C}_0\to X$ and take the complex $\bR g_* \cO_{\wt{C}_0}.$
Since $g$ is proper being a composition of proper morphisms, the complex $\bR g_* \cO_{\wt{C}_0}$ is a cohomologically bounded complex with coherent cohomology.
From Theorem 4.1 of \cite{LN} we know that, given any integer $m,$ we may find a perfect complex $P^{\cdot}$ and a morphism $u: P^{\cdot}\to
\bR g_* \cO_{\wt{C}_0}$ so that  the induced morphisms $\H^i(u)$ on cohomology sheaves are isomorphisms for all
$i>m.$ Let $m=-k-2$ where $X$ can be covered by $k$ affine open sets. Choose $u: P^{\cdot}\to
\bR g_* \cO_{\wt{C}_0}$ as above and complete to an exact triangle
\[
P^{\cdot}\stackrel{u}{\lto}\bR g_* \cO_{\wt{C}_0}\lto Q\lto P[1].
\]
Since all cohomology sheaves $\H^i(Q)$ are trivial when $i>-N-2$ the map between $\Hom(\cO_X, P^{\cdot})$ and $\Hom(\cO_X, \bR g_* \cO_{\wt{C}_0})$
is an isomorphism. Thus we obtain
\[
\Hom_X(\cO_X, P^{\cdot})\cong \Hom_X(\cO_X, \bR g_* \cO_{\wt{C}_0})\cong \Hom_{C_0}(\cO_{\wt{C}_0}, \cO_{\wt{C}_0})\cong H^0(\wt{C}_0, \cO_{\wt{C}_0})=A,
\]
where $\Spec A=\wt{C}_0.$
The  $\kk$\!-space $A$ is infinite dimensional over $\kk,$ while $\cO_X$ an $P^{\cdot}$ are both perfect.
Hence, $\prf X$ can not be proper. This proves the proposition.
\end{proof}
We recall that a scheme of finite type over a field $\kk$ is called {\em smooth}  if the scheme
$\widebar{X}=X\otimes_{\kk} \widebar{\kk}$ is regular, where $\widebar{\kk}$ is an algebraic closure
of $\kk.$

\begin{proposition}\label{smooth_proper} Let $X$ be a separated scheme of finite type
over an arbitrary field. Then $X$ is smooth and proper if and only if the DG category
$\prfdg X$ is smooth and proper.
\end{proposition}
\begin{proof}
The statement for smoothness (without properness) is proved in \cite[3.13]{Lu} for a separated scheme of finite type over a perfect field.
On the other hand, the definition of a smooth scheme and a smooth DG category is invariant under a base field change.
Indeed, if $\prfdg X$ is smooth, then $\prfdg \widebar{X}$ is smooth and, hence, by \cite[3.13]{Lu} the scheme $\widebar{X}$ is smooth (regular).
But this is exactly smoothness of $X$ by definition.
The properness of $X$ is proved in Proposition \ref{prop_scheme}.

Now if $X$ is smooth and proper, then properness of the DG category
$\prfdg X$ follows from Proposition \ref{proper}.
Since $\widebar{X}$ is regular we obtain that  the DG category $\prfdg \widebar{X}$ is smooth by \cite[3.13]{Lu}.
Finally, we should argue that smoothness of $\prfdg\widebar{X}$ implies
smoothness of $\prfdg X.$ Since $\kk\subset\widebar{\kk}$ is faithfully flat,  the following property holds:
for any DG category $\dA$ and any $\dA$\!--module
$\mM$ if $\mM\otimes_{\kk}\widebar{\kk}$ is perfect as $\dA\otimes_{\kk}\widebar{\kk}$\!--module, then
$\mM$ is also perfect.
\end{proof}

\section{Gluing of smooth projective schemes and geometric noncommutative schemes}

\subsection{Geometric noncommutative schemes}

Let $X$ and $Y$ be two smooth projective schemes over a field $\kk.$
Consider DG categories of perfect complexes $\prfdg X$ and $\prfdg Y.$ Since $X$ and $Y$
are smooth these categories are quasi-equivalent to DG categories $\dD^b(\coh X)$ and $\dD^b(\coh Y),$
respectively. Theorem \ref{dg_saturated} tells us that for a regular and proper $X$ there is a quasi-equivalence
\[
\dR\lHom(\prfdg X^{\op},\; \prfdg\kk)\cong \prfdg X.
\]
Therefore, applying the canonical quasi-equivalence (\ref{RHom}) we obtain that
\begin{multline}
\dR\lHom(\prfdg Y\otimes_{\kk}\prfdg X^{\op},\; \prfdg\kk)\cong \dR\lHom(\prfdg Y, \dR\lHom(\prfdg X^{\op},\; \prfdg\kk))
\cong\\
 \dR\lHom(\prfdg Y,\; \prfdg X).
\end{multline}
Moreover, there is the following theorem due to B.~To\"en.
\begin{theorem}\cite{To}\label{toen}
Let $X$ and $Y$ be smooth projective schemes over a field $\kk.$
Then there is a canonical isomorphism in $\Hqe$
\[
\dR\lHom(\prfdg Y,\; \prfdg X)\cong\prfdg(X\times_k Y).
\]
In particular, the DG category $\prfdg(X\times_k Y)$ is quasi-equivalent to the DG category of perfect
DG modules over $\prfdg Y^{\op}\otimes_{\kk}\prfdg X.$
\end{theorem}

This quasi-equivalence can be described explicitly.
As was explained in Section \ref{3.1} there are DG functors
\[
\mpr_1^*: \prfdg X\lto \prfdg(X\times Y),\quad \text{and}\quad \mpr_2^*: \prfdg Y\lto \prfdg(X\times Y)
\]
For any perfect complex $\mE^{\cdot}$ on the product $X\times_{\kk} Y$ we can define a bimodule
$\mS_{\mE^{\cdot}}$ by the rule
\[
\mS_{\mE^{\cdot}}(B, A)\cong \dHom_{\prfdg(X\times Y)}(\mpr_1^* A,\; \mpr_2^* B\otimes \mE),
\quad\text{where}
\quad
A\in \prfdg X,\; B\in\prfdg Y.
\]
This is exactly the quasi-equivalence between the DG category $\prfdg(X\times Y)$ and
the DG category of perfect $\prfdg Y\hy\prfdg X$\!--bimodules, i.e. perfect
DG modules over $\prfdg Y^{\op}\otimes_{\kk}\prfdg X.$

Let $\M\subset\prf X$ and $\N\subset\prf Y$ be admissible subcategories, where $X$ and $Y$ are smooth projective schemes
over $\kk.$ Consider the induced DG subcategories $\dM\subset\prfdg X$ and $\dN\subset\prfdg Y$
and  the induced DG functor $\mF: \dN^{\op}\otimes\dM\to \prfdg Y^{\op}\otimes\prfdg X$ that is fully faithful.
This DG functor gives the extension quasi-functor
\[
\mF^*:\prfdg(\dN^{\op}\otimes\dM)\lto \prfdg (X\times Y)
\]
that is fully faithful on the homotopy categories by Proposition \ref{Keller2}.
In more detail,  for any pair of admissible subcategories $\M\subset\prf X$ and $\N\subset\prf Y$ we can define
a full triangulated subcategory $\M\prd\N$ of the category
$\prf(X\times Y)$ as the minimal triangulated subcategory
of $\prf(X\times Y)$ closed under taking direct summands and containing
all objects of the form $\pr_1^* M\otimes \pr_2^* N$ with $M\in \M$ and $N\in \N.$
Denote by $\dM\prd\dN\subset \prfdg (X\times Y)$ the induced enhancement of $\M\prd\N.$

It is easy to see that $\prfdg(\dN^{\op}\otimes\dM)$ is quasi-equivalent to $\dM\prd\dN$ because
$\prf(\dN^{\op}\otimes\dM)$ and $\M\prd\N$ are classically generated by $\Ob(\dN^{\op}\otimes\dM).$

Being admissible subcategories  in DG categories of perfect complexes on smooth and proper schemes,
the DG categories $\dM$ and $\dN$
are smooth and proper (see Theorem \ref{smooth_glue} for smoothness and Proposition \ref{reg_prop} for properness). By Theorem \ref{dg_saturated} there is a quasi-equivalence
\[
\dR\lHom(\dM^{\op},\; \prfdg\kk)\cong \dM.
\]
Therefore, applying the canonical quasi-equivalence (\ref{RHom}) we obtain that
\[
\dR\lHom(\dN\otimes_{\kk}\dM^{\op},\; \prfdg\kk)\cong \dR\lHom(\dN, \dR\lHom(\dM^{\op},\; \prfdg\kk))
\cong\dR\lHom(\dN,\; \dM).
\]

Let us summarize what we have.
\begin{proposition}\label{geometric}
Let $X$ and $Y$ be two smooth projective schemes and $\dM\subset\prfdg X$ and $\dN\subset\prfdg Y$ be full DG subcategories such that
the subcategories $\M=\Ho(\dM)$ and $\N=\Ho(\dN)$ are admissible in $\prf X$ and $\prf Y,$ respectively.
In this case there are quasi-equivalences of DG categories
\[
\dR\lHom(\dN,\; \dM)\cong\prfdg(\dN^{\op}\otimes_{\kk}\dM)\cong \dM\prd\dN\subset\prfdg(X\times_{\kk}Y),
\]
where $\dM\prd\dN$ is a full DG subcategory of $\prfdg(X\times Y)$ that is classically generated by objects of the form
$\pr_1^* M\otimes \pr_2^* N$ with $M\in \M$ and $N\in \N.$
\end{proposition}

We are interested in smooth (or regular) and proper noncommutative scheme $\prfdg\E.$
Smooth and proper geometric noncommutative schemes naturally appear as induced enhancements of
admissible subcategories $\N\subset\prf X$ for some smooth and projective scheme $X.$

\begin{definition}
A noncommutative scheme $\prfdg\E$ (see Definition \ref{noncommutative_scheme})
will be called a {\em geometric noncommutative scheme} if
there are a smooth and projective scheme $X$ and an admissible subcategory $\N\subset \prf X$
such that $\prfdg \E$ is quasi-equivalent to the corresponding enhancement $\dN\subset \prfdg X$ of $\N.$
\end{definition}

We can consider a 2-category of smooth and proper noncommutative schemes
$\RPNS$ over a field $\kk.$
Objects of $\RPNS$ are DG categories $\dA$ of the form $\prfdg\dE,$ where
$\dE$ is a smooth and proper DG algebra; 1-morphisms are
quasi-functors $\mT;$  2-morphisms are morphisms of quasi-functors, i.e. morphisms in
$\D(\dA^{op}\otimes\dB).$
The 2-category $\RPNS$ has a natural full 2-subcategory of geometric noncommutative schemes $\GNS.$
Evidently, $\GNS$ contains all smooth and proper commutative schemes with $\prf(X\times Y)$ as category of morphisms
between $X$ and $Y.$
The natural question that arises is following.

\begin{question}
Is there a smooth and proper noncommutative scheme that is not geometric?
\end{question}

The first attempt to find such a noncommutative scheme is to glue geometric noncommutative schemes via a bimodule.
Another way is to consider  a finite dimensional $\kk$\!--algebra $\Lambda$ of finite global dimension and take
the DG category $\prfdg\Lambda.$

The main goal of this paper is to show that these two approaches do not lead  us to new noncommutative schemes. We show that the world of geometric noncommutative schemes is closed under gluing via any perfect bimodule.
More precisely, consider smooth and proper geometric noncommutative schemes $\prfdg\dE_1$ and $\prfdg\dE_2$ such that
$\prfdg\dE_i$ is quasi-equivalent to $\dN_i\subset\prfdg X_i,$ where $X_i$ are smooth and projective and
$\N_i=\Ho(\dN_i)$ are admissible in $\prf X_i,$ respectively.
After that we take a gluing $\prfdg\dE_1\underset{\mS}{\oright}\prfdg\dE_2$ via a perfect bimodule $\mS$ and show that
the resulting noncommutative scheme is geometric too (see Theorem \ref{main2}).

\begin{remark}\label{Hironaka}
{\rm
We work with smooth and projective schemes. On the other hand, over a field of characteristic 0
the category of perfect complexes on any smooth and proper scheme can be realized as an admissible subcategory
in a smooth and projective scheme. Indeed, by Chow's Lemma for a proper scheme $X$ there is a proper birational morphism
$f: Y\to X$ from a projective scheme $Y.$ Now applying Hironaka hut for resolution of the birational map $X\dashrightarrow Y,$
we can find a proper scheme $Z$ with birational maps to $X$ and $Y$ such that the morphism $\pi: Z \to X$
is a sequence of blowups with regular centers. This implies that $Z$ is smooth and also projective because there is a proper
birational morphism to the projective scheme $Y.$ Finally, the inverse image functor $\bL\pi^*$ gives a full embedding
of $\prf X$ into $\prf Z.$
Note that over the complex numbers $\CC$ we can apply a result of Moishezon asserting that any smooth and proper algebraic space over $\CC$
becomes a projective variety after some blowups along smooth centers.
}
\end{remark}

We also show that for any finite dimensional $\kk$\!--algebra $\Lambda$ such that its semisimple part
$S=\Lambda/\rd$ is separable over $\kk$ there are a smooth and projective scheme $X$ and a perfect complex
$\E^{\cdot}$ such that $\bR\Hom(\E^{\cdot},  \E^{\cdot})\cong \Lambda.$
This implies that for the finite dimensional algebra $\Lambda$ the DG category
$\prfdg\Lambda$ is quasi-equivalent to a full DG subcategory of $\prfdg X$ and in the case of finite global dimension
the smooth and proper noncommutative scheme $\prfdg\Lambda$ is geometric (Theorem \ref{algebra}).

\subsection{Perfect complexes as direct images of line bundles}
Let $X$ be a scheme and $\E^{\cdot}$ be a strict perfect complex, i.e. a bounded complex of algebraic vector bundles
(locally free sheaves of finite type). In this section we show that
any such a strict perfect complex $\E^{\cdot}$ can be realized as a direct image of a line bundle with respect to a smooth morphism
$Z\to X.$

\begin{proposition}\label{line} Let $X$ be a scheme. Let $\E^{\cdot}=\{\E^0\to\cdots\to\E^k\}$
be a bounded complex of algebraic vector bundles on $X.$
Then there are a scheme $Z$ with a morphism $f: Z\to X$ and a line bundle $\L$ on $Z$
such that
\begin{enumerate}
\item[1)]
$\bR f_{*}\L\cong \E^{\cdot}$ in the derived category $\D(\Qcoh X);$
\item[2)] $\bR f_{*}\L^{-1}\cong 0;$
\item[3)] the morphism $f: Z\to X$ is a composition of maps $Z=X_n\to X_{n-1}\to\cdots\to X_0=X,$
where each $X_{p+1}$ is the projectivization of a vector bundle $\F_p$ over $X_p.$
\end{enumerate}
\end{proposition}
\begin{dok}
First, we should note that the following construction does not work for the $0$\!--complex and for
a complex $\E^{\cdot}$ that is a line bundle. However, it can be easily improved. In such cases we
change $\E^{\cdot}$ to a quasi-isomorphic complex by adding an acyclic complex of the form $\F\stackrel{\id}{\to}\F.$

Now we will prove the proposition by induction on the length of the complex $\E^{\cdot}.$
If $\E^{\cdot}$ has only one term and it is a vector bundle $\E$ (of rank $>1$),  then we take the line bundle $\L=\cO(1)$ on the
projective bundle $f: \PP(\E^{\vee})\to X.$ As result we obtain that $\bR f_* \L= f_*\L\cong\E$
and $\bR f_* \L^{-1}=0.$

Assume that $\E^{\cdot}=\{\E^0\stackrel{d}{\lto}\E^1\}$ is a complex of vector bundles that has only two nontrivial terms in degree $0$ and $1.$
Taking $f_1: X_1=\PP(\E^{1\vee})\to X$ and $\L_1=\cO(1)$ we obtain that $\bR f_{1*}\L\cong\E^1$ as described above.
Now let $X_2=X_1\times\PP^1$ and $f_2$ is the projection on $X.$ Take $\L_2=\L_1\boxtimes\cO(-2).$
It is easy to see that $\bR \pr_{1*}\L_2$ has only one nontrivial term $\bR^1 \pr_{1*}\L_2$ that is isomorphic to $\L_1.$
Hence $\bR f_{2*}\L_2\cong \E^1[-1].$
Thus we obtain a sequence of isomorphisms
\begin{multline*}
\Ext^1_{X_2}(f_2^*\E^0, \L_2)\cong \Ext^1_{X_1}(f_1^*\E^0, \bR \pr_{1*}\L_2)\cong \Hom_{X_1}(f_1^*\E^0, \L_1)\cong\Hom_{X}(\E^0, \bR f_{1*}\L_1)\cong\\
\Hom_{X}(\E^0, \E^1).
\end{multline*}
Under this isomorphism the differential $d$ induces an element $e\in \Ext^1(f_2^*\E^0, \L_2).$
Let us consider the extension
\[
0\lto\L_2\lto\F\lto f_2^*\E^0\lto 0
\]
given by the element $e.$ Applying the functor $\bR f_{2*}$ to this short exact sequence we obtain an exact triangle of the form
\[
\bR f_{2*} \F\lto \E^0 \stackrel{\alpha}{\lto} \bR f_{2*}\L_2[1].
\]
By construction, $\bR f_{2*}\L_2[1]\cong \E^1$ and $\alpha=d.$ Therefore, $\bR f_{2*} \F$ is isomorphic to the complex $\E^{\cdot}.$
Finally, we consider $Z=\PP(\F^{\vee})$ with the natural morphism $f$ to $X$ and $\L=\cO_{\PP(\F^{\vee})}(1).$ We get that $\bR f_{*}\L\cong\E^{\cdot}.$
Since the rank of $\F$ is bigger than one, it is also evident that $\bR f_* \L^{-1}=0.$

The same trick works for any complex of vector bundles
\[
\E=\{\E^0\stackrel{d}{\lto}\E^1\lto\cdots\lto\E_k\}.
\]
Indeed, consider the stupid truncation $\sigma^{\ge 1}\E^{\cdot}.$
By induction we can assume that there is $f_{n-1}: X_{n-1}\to X$ and $\L_{n-1}$ on $X_{n-1}$ such that
$\bR f_{(n-1)*}\cong \sigma^{\ge 1}\E^{\cdot}[1].$

Now repeat the procedure described above.
Let $X_{n}=X_{n-1}\times\PP^1$ and $f_{n}$ be the projection on $X.$
Take $\L_{n}=\pr_1^*\L_{n-1}\boxtimes\cO(-2).$
We have $\bR f_{n*}\L_{n}\cong \sigma^{\ge 1}\E^{\cdot}.$
There is an isomorphisms
\[
\Ext^1(f_{n}^*\E^0, \L_{n})\cong \Ext^1(\E^0, \bR f_{n*}\L_n)\cong
\Hom(\E^0, \sigma^{\ge 1}\E^1[1])
\]
Under this isomorphism the differential $d:\E_0\to \sigma^{\ge 1}\E^1[1]$ induces an element $e\in \Ext^1(f_{n}^*\E^0, \L_{n}).$
Let us consider the extension
\[
0\lto\L_n\lto\F\lto f_n^*\E^0\lto 0
\]
given by the element $e.$ Applying the functor $\bR f_{n*}$ to this short sequence we obtain an exact triangle of the form
\[
\bR f_{n*} \F\lto \E^0 \stackrel{\alpha}{\lto} \bR f_{n*}\L_n[1].
\]
By construction, $\bR f_{n*}\L_2[1]\cong \sigma^{\ge 1}\E^1[1]$ and $\alpha=d.$
Therefore, $\bR f_{n*} \F$ is isomorphic to the complex $\E^{\cdot}.$
Finally, we consider $Z=\PP(\F^{\vee})$ with the natural morphism $f$ to $X$ and $\L=\cO(1).$
We get that $\bR f_{*}\L\cong\E^{\cdot}.$ By construction, the scheme $f:Z\to X$ is a sequence of
projective bundles. Moreover, we have $\bR f_*\L^{-1}=0,$ because the rank of $\F$ is bigger than one and $\bR p_*\L^{-1}=0,$
where $p$ is the projection of $Z=\PP(\F^{\vee})$ to $X_n.$
\end{dok}
\begin{remark}
{\rm
Assume that a quasi-compact and separated scheme $X$ has enough locally free sheaves, i.e. for any quasi-coherent sheaf of finite type
$\F$ there is an algebraic vector bundle $\E$ on $X$ and an epimorphism $\E\twoheadrightarrow\F.$
In this case, any perfect complex is quasi-isomorphic to
a strict perfect complex (see \cite[2.3.1]{TT}) and, hence, Proposition \ref{line} can be applied to any perfect complex up to a shift
in the triangulated category.
Note that any quasi-projective scheme and any separated regular noetherian scheme have enough locally free sheaves.
}
\end{remark}

\subsection{Blowups and gluing of smooth projective schemes}

Let $X_1$ and $X_2$ be two smooth irreducible projective schemes. Let $\E^{\cdot}$ be a perfect complex
on the product $X_1\times X_2.$ Since $X_i$ are projective any perfect complex on $X_1\times X_2$ is globally (not only locally) quasi-isomorphic to
 a strictly perfect complex, i.e. a bounded complex of locally free sheaves of finite type (see, e.g. \cite[2.3.1]{TT}).
A strictly perfect complex will be also called a bounded complex of vector bundles.

Applying a shift in the triangulated category we can assume that  $\E^{\cdot}\in\prf (X_1\times X_2)$ is
a complex of the form $\{\E^0{\to}\E^1\to\cdots\to\E_k\}.$
By Proposition \ref{line} there is a scheme $Z$ with a morphism $f: Z\to X_1\times X_2$
and a line bundle $\L$ on $Z$ such that $\bR f_* \L\cong \E^{\cdot}.$
Let us fix such $Z$ and $f.$ By the construction of $Z,$ since $X_1$ and $X_2$ are smooth projective the scheme $Z$ is also
smooth and projective and the morphism $f$ is smooth.

Denote by $q_1$ and $q_2$ the canonical morphisms from $Z$ to $X_1$ and $X_2$ respectively.
Fix very ample line bundles $\M_1$ and $\M_2$ on $X_1$ and $X_2$ respectively.
Using Serre's theorem we can find a very ample line bundle $\L'$ on $Z$ such that
the three line bundles $\L_1=q_1^*\M_1^{-1}\otimes\L',$ $\L_2=q_2^*\M_2^{-1}\otimes\L'$ and $
\L_3=\L^{-1}\otimes \L'$ are very ample as well.

Denote by $s_1, s_2, s_3$ the closed immersions of $Z$ to projective spaces
$\PP^{n_1}, \PP^{n_2}, \PP^{n_3}$ induced by $\L_1, \L_2,$ and $\L_3$ respectively.
The product map $i_1=(q_1, s_1, s_3)$ gives a closed immersion of $Z$ to the projective scheme
$P_1=X_1\times \PP^{n_1}\times\PP^{n_3}.$ Similarly, we obtain a closed immersion $i_2=(q_2, s_2, s_3)$ of $Z$ to
$P_2=X_2\times \PP^{n_2}\times\PP^{n_3}.$ It directly follows from construction
that there are isomorphisms
\begin{equation}\label{line_bundles}
\begin{aligned}
& i_1^* (\M_1\boxtimes\cO(1)\boxtimes\cO(-1)) \cong  i_2^* (\M_2\boxtimes\cO(1)\boxtimes\cO(-1)) \cong  \L,\\
& i_1^* (\M_1\boxtimes\cO(1)\boxtimes\cO(1)) \cong i_2^* (\M_2\boxtimes\cO(1)\boxtimes\cO(1)) \cong  \L'\otimes\L_3
\end{aligned}
\end{equation}
Consider these two closed immersions $i_1: Z\to P_1$ and $i_2: Z\to P_2.$ It is known that the gluing $P_1\bigsqcup_{Z} P_2$
of $P_1$ and $P_2$ along $Z$ is a scheme (\cite[3.9]{Schwede} or \cite[5.4]{F}). It has two irreducible components that meet
along $Z.$ Denote this gluing by $T.$ The scheme $T$ is a pushout (fibred coproduct) in the category
of schemes. In our case we also can argue that the scheme $T$ is projective.
\begin{lemma} The scheme $T=P_1\bigsqcup_{Z} P_2$ is projective.
\end{lemma}
\begin{dok} Consider very ample line bundles $\M_1\boxtimes\cO(1)\boxtimes\cO(1)$ and $\M_2\boxtimes\cO(1)\boxtimes\cO(1)$
on $P_1$ and $P_2.$ Since their restrictions on $Z$ are isomorphic to $\L'\otimes\L_3$ we can glue them into a line bundle
$\N$ on $T.$
It follows from \cite[2.6.2]{EGA3} that $\N$ is ample on $T$ (see also \cite[6.3]{F}).
\end{dok}

Consider a closed immersion $j: T\hookrightarrow \PP^N$ to a projective space $\PP^N.$
Denote by $j_1, j_2$ the induced closed immersions of $P_1, P_2$ to $\PP^N.$
Consider the blowup $V'$ of $\PP^N$ along $P_1.$ Take the strict transform $\wt{P}_2$ of $P_2$ in $V'.$
It is the blowup of $P_2$ along the subvariety $Z.$
Denote by $V$ the blowup of $V'$ along $\wt{P}_2.$
This construction can be illustrated by the following diagram (\ref{diagram})

\begin{align}\label{diagram}
\xymatrix{
& \wt{E}_1 \ar[d]_{\rho} \ar@{^{(}->}[r]^{\wt{e}_1} & V \ar[d]_{\pi} & E_2
\ar@{_{(}->}[l]_{e_2} \ar[d]^{h_2}
\\
D \ar[d]_{g} \ar@{^{(}->}[r]^{d_1}& E_1 \ar[d]_{h_1} \ar@{^{(}->}[r]^{e_1} &
V' \ar[d]_{\pi'} & \wt{P}_2 \ar[d]^{\tau} \ar@{_{(}->}[l]_{\wt{j}_2}
 & D \ar[d]^{g} \ar@{_{(}->}[l]_{d_2}
\\
Z \ar@{^{(}->}[r]^{i_1} \ar[dr]_{q_1} & P_1 \ar[d]^{p_1} \ar@{^{(}->}[r]^{j_1} & \PP^N & P_2 \ar[d]_{p_2} \ar@{_{(}->}[l]_{j_2} &
Z \ar@{_{(}->}[l]_{i_2} \ar[dl]^{q_2}
\\
& X_1 && X_2
}
\end{align}
where $h_1: E_1\to P_1$ is the exceptional divisor of the first blowup, $g: D\to Z$ is the exceptional divisor
of the induced blowup $\tau: \wt{P}_2\to P_2,$ $\wt{E}_1$ is the strict transformation of the divisor $E_1$ under the second blowup,
and $h_2: E_2\to \wt{P}_2$ is the exceptional divisor
of the second blowup.

Since we started from smooth projective schemes $X_1$ and $X_2,$ we obtain smooth and projective schemes $Z,$ $P_1$ and $P_2.$
A blowup of a smooth projective scheme along a smooth closed subscheme  brings to a smooth and projective scheme (see, e.g.
\cite[Th.8.1.19]{Liu}). Thus, we obtain

\begin{lemma}
All schemes in diagram (\ref{diagram}) are projective and smooth.
\end{lemma}

Now let us analyze morphisms in our diagram (\ref{diagram}) and functors induced by them.

\begin{proposition}\label{functors}
In the diagram (\ref{diagram}) the following properties of morphisms hold:
\begin{enumerate}
\item[1)] the morphisms $g, h_1, h_2, p_1, p_2$ are projectivizations of vector bundles, and the functors $g^*, h_1^*, h_2^*, p_1^*, p_2^*$
are fully faithful;
\item[2)] the morphisms $\pi', \pi, \tau, \rho$ are blowups along smooth centers, and the exact functors
$\bL {\pi'}^{*},$ $\bL\pi^*,$ $\bL\tau^*,$ and $\bL\rho^*$ are fully faithful;
\item[3)] functors of the form $\bR d_{2*}(\K\otimes g^*(-)),\; \bR e_{1*}(\K\otimes h_1^*(-)),\; \bR e_{2*}(\K\otimes h_2^*(-)),$
where $\K$ is a line bundle on $D, E_1, E_2,$ respectively, are fully faithful;
\item[4)] the functors $\bR d_{2*}, \bR e_{1*}, \bR e_{2*}$ have right adjoints
$d_2^{\flat}, e_1^{\flat}, e_2^{\flat}$ and there are isomorphisms
\[\hspace*{1cm}
d_2^{\flat}\cong \bL d_2^*(\cO(D)\otimes (-))[-1],\quad
e_1^{\flat}\cong \bL e_1^*(\cO(E_1)\otimes (-))[-1],\quad
e_2^{\flat}\cong \bL e_2^*(\cO(E_2)\otimes (-))[-1].
\]
\end{enumerate}
\end{proposition}

\begin{dok}
1) and 2) follow from the construction and the projection formula, because the derived direct image of the structure sheaf
under each of these morphisms is isomorphic to the structure sheaf of a target.
3) is proved in \cite[4.2]{Blow} or \cite[2.2.7]{O2}.
4) follows from the fact that for any closed immersion $i$ of locally a complete intersection $Z$ to $Y$
the right adjoint $i^{\flat}$ to $\bR i_{*}$ has
the form $\bL i^*(\cdot) \otimes \omega_{Z/Y}[-r],$ where
$\omega_{Z/Y}\cong \Lambda^{r} N_{Z/Y}$ and $r$ is the codimension.
\end{dok}

\begin{theorem}\label{main}
Let $X_1$ and $X_2$ be smooth irreducible projective schemes and let $\E^{\cdot}$ be a perfect complex
on the product $X_1\times X_2.$ Let $V$ be a smooth projective scheme constructed above.
Then the DG category $\prfdg X_1 \underset{\E^{\cdot}}{\oright}\prfdg X_2$ is quasi-equivalent to a full DG subcategory
of $\prfdg V$ and, hence,
the triangulated category $\Ho(\prfdg X_1 \underset{\E^{\cdot}}{\oright}\prfdg X_2)$ is admissible in
$\prf V.$
\end{theorem}

Consider the DG categories $\prfdg X_1$ and $\prfdg X_2$ and the following composition quasi-functors
\begin{equation}
\mPhi:=\mpi^* \me_{1*}( \cO_{E_1}(E_1)\otimes \mh_1^* \mp_1^*(-)), \quad\text{and}\quad
\mPsi:=\me_{2*} \mh_2^* \mtau^* (\mp_2^*(-)\otimes \R)[1]
\end{equation}
from $\prfdg X_1$ and $\prfdg X_2$ to $\prfdg V$ respectively, where $\cO_{E_1}(E_1)$ is the restriction
of the line bundle $\cO(E_1)$ from $V'$ to $E_1,$ and $\R\cong \M_2\boxtimes\cO(1)\boxtimes\cO(-1)$ is a line bundle on $P_2.$
These quasi-functors induce exact composition functors
\begin{equation}\label{funct}
\Phi:=\bL\pi^* \bR e_{1*}( \cO_{E_1}(E_1)\otimes h_1^* p_1^*(-)), \quad\text{and}\quad
\Psi:=\bR e_{2*} h_2^* \bL\tau^* (p_2^*(-)\otimes \R)[1]
\end{equation}
from the triangulated categories $\prf X_1$ and $\prf X_2$ to the triangulated category $\prf V.$

\begin{lemma}\label{lemma_orthogonal}
The functors $\Phi$ and $\Psi$ are fully faithful and the subcategories
$\Phi(\prf X_1)$ and $\Psi(\prf X_2)$ are semi-orthogonal so that
$\Phi(\prf X_1)$ is in the right orthogonal $\Psi(\prf X_2)^{\perp}.$
\end{lemma}
\begin{dok} The functors $\Phi$ and $\Psi$ are fully faithful as  compositions
of fully faithful functors. It follows from 1)-3) of Proposition \ref{functors}.

Semi-orthogonality is a consequence of the fact that $\bL\pi^*(\prf V')$ is in
the right orthogonal $\bR e_{2*} h_2^*(\prf \wt{P}_2)^{\perp}.$
The last statement follows from the chain of isomorphisms
\begin{multline*}
\Hom(\bR e_{2*} h_2^* B,\; \bL\pi^* A)\cong
\Hom(h_2^* B,\; e_2^{\flat}\bL\pi^* A)\cong
\Hom(h_2^* B,\; \bL e_2^*\bL\pi^* A\otimes \cO_{E_2}(E_2))\cong\\
\Hom(h_2^* B,\; h_2^*\bL \wt{j}_2^* A\otimes \cO_{E_2}(E_2))\cong
\Hom(B,\; \bL \wt{j}_2^* A\otimes \bR h_{2*}\cO_{E_2}(E_2))=
0
\end{multline*}
where $A\in\prf V', B\in\prf \wt{P}_2.$
The last equality holds because $\cO_{E_2}(E_2)\cong\cO_{E_2}(-1)$ under consideration of $E_2$ as the projectivization of the normal bundle
of $\wt{P}_2$ in $V',$ i.e. we have $\bR h_{2*} \cO_{E_2}(E_2)=0.$
\end{dok}

\begin{proposition}\label{main_prop} Let $X_1$ and $X_2$ be smooth projective schemes and let $\E^{\cdot}$ be a perfect complex
on the product $X_1\times X_2.$ Let $V$ be the smooth projective scheme constructed above and let $\Phi, \Psi$ be exact functors
defined by formula (\ref{funct}).
Let $\F^{\cdot}$ and $\G^{\cdot}$ be perfect complexes on $X_1$ and $X_2,$ respectively.
Then there is an isomorphism
\[
\Hom_{V}(\Phi(\F^{\cdot}),\; \Psi(\G^{\cdot}))\cong \Hom_{X_1\times X_2}(\pr_1^* \F^{\cdot}, \; \pr_2^*\G^{\cdot}\otimes \E^{\cdot}),
\]
where $\pr_i$ denote the projections of $X_1\times X_2$ on $X_i.$
\end{proposition}

\begin{dok}
Firstly, consider objects $A\in\prf V'$ and $B\in\prf \wt{P}_2.$ There is a sequence of isomorphisms
\begin{multline}\label{1_sequence}
\Hom_V(\bL\pi^* A, \; \bR e_{2*} h_2^* B)
\cong
\Hom_{V'}(A, \; \bR\pi_*\bR e_{2*} h_2^* B)
\cong\\
\Hom_{V'}(A, \; \bR\wt{j}_{2*}\bR h_{2*} h_2^* B)
\cong
\Hom_{V'}(A, \; \bR\wt{j}_{2*} B).
\end{multline}

Secondly, take $A\in\prf P_1$ and $B\in\prf P_2.$ Consider the commutative square
\begin{equation}\label{square}
\xymatrix{
D \ar@{^{(}->}[r]^{d_2} \ar[d]_{d_1} & \wt{P}_2 \ar[d]^{\wt{j}_2}\\
E_1 \ar@{^{(}->}[r]^{e_1} & V'
}
\end{equation}
that is a part of our main diagram (\ref{diagram}). The commutative square (\ref{square}) is cartesian.
Moreover, it is Tor-independent. This means that $\Tor^p_{\cO_{V'}}(\cO_{E_1}, \cO_{\wt{P}_2})=0$ for all $p>0.$
Therefore, by \cite[IV 3.1]{SGA6} or \cite[2.5.6]{TT} there is a canonical base change isomorphism of functors
\[
\bL e_1^{*}\bR \wt{j}_2\stackrel{\sim}{\lto}\bR d_{1*}\bL d_2^*.
\]

Using this isomorphism of functors we obtain the following sequence of isomorphisms
\begin{equation}
\begin{split}
&\Hom_{V'}(\bR e_{1*} (\cO_{E_1}(E_1)\otimes h_1^* A), \; \bR \wt{j}_{2*} \bL\tau^*  B)
\cong
\Hom_{E_1}(\cO_{E_1}(E_1)\otimes h_1^* A, \; e_1^{\flat}\bR \wt{j}_{2*} \bL\tau^*  B)
\cong\\
\label{2_sequence}
&\Hom_{E_1}(h_1^* A, \; \bL e_1^{*}\bR \wt{j}_{2*} \bL\tau^*  B[-1])
\cong
\Hom_{E_1}(h_1^* A, \; \bR d_{1*}\bL d_2^* \bL\tau^*  B[-1])
\cong\\
&
\Hom_{D}(\bL d_1^* h_1^* A, \; \bL d_2^* \bL\tau^*  B[-1])
\cong
\Hom_{D}(g^* \bL i_1^* A, \; g^* \bL i_2^*  B[-1])
\cong
\Hom_Z(\bL i_1^* A, \; \bL i_2^*  B[-1])
\end{split}
\end{equation}

Now combining (\ref{1_sequence}) and (\ref{2_sequence}), we obtain
\begin{equation}\label{3_sequence}
\begin{split}
&\Hom_{V}(\Phi(\F^{\cdot}),\; \Psi(\G^{\cdot}))
\cong
\Hom_{V'}(\bR e_{1*} (\cO_{E_1}(E_1)\otimes h_1^* p_1^* \F^{\cdot}), \; \bR \wt{j}_{2*} \bL\tau^* (p_2^* \G^{\cdot}\otimes \R)[1])
\cong\\
&\Hom_{Z}(\bL i_1^* p_1^* \F^{\cdot}, \; \bL i_2^* (p_2^* \G^{\cdot}\otimes \R) )
\cong
\Hom_{Z}(q_1^* \F^{\cdot}, \; q_2^* \G^{\cdot}\otimes \L)).
\end{split}
\end{equation}

The last isomorphism is a consequence of the construction of $P_2$ and the line bundle $\R$ on $P_2.$
By (\ref{line_bundles}) the restriction of $\R$ on $Z$ coincides with the line bundle $\L.$

Finally, we have $q_i=\pr_i\cdot f$ for $i=1,2$ and we know that $\bR f_*\L\cong\E^{\cdot}$ by construction
from Proposition \ref{line}. This implies
\begin{multline}\label{4_sequence}
\Hom_{Z}(q_1^* \F^{\cdot},\; q_2^* \G^{\cdot}\otimes \L))\cong
\Hom_{X_1\times X_2}(\pr_1^* \F^{\cdot},\; \pr_2^* \G^{\cdot}\otimes \bR f_*\L))
\cong\\
\Hom_{X_1\times X_2}(\pr_1^* \F^{\cdot},\; \pr_2^* \G^{\cdot}\otimes \E^{\cdot})).
\end{multline}
The isomorphisms (\ref{3_sequence}) and (\ref{4_sequence}) finish the proof of the proposition.
\end{dok}

\bigskip
\noindent{\bf Proof of Theorem \ref{main}.}\;
Let us consider the DG functors
\[
\mPhi: \prfdg X_1\lto \prfdg V,\quad \text{and}\quad \mPsi: \prfdg X_2\lto \prfdg V.
\]
They induce a bimodule
$\mS$ determined by the following rule
\begin{equation}\label{s_bimodule}
\mS(B, A)\cong \dHom_{\prfdg V}(\mPhi A,\; \mPsi B),
\quad\text{where}
\quad
A\in \prfdg X_1,\; B\in\prfdg X_2.
\end{equation}

On the other hand, the calculations from Proposition \ref{main_prop} gives us that the bimodule $\mS$ is quasi-isomorphic
to a bimodules $\mS_{\E^{\cdot}}$ given by the rule
\[
\mS_{\E^{\cdot}}(B, A)\cong \dHom_{\prfdg(X\times Y)}(\mpr_1^* A,\; \mpr_2^* B\otimes \E^{\cdot}),
\quad\text{where}
\quad
A\in \prfdg X_1,\; B\in\prfdg X_2.
\]

Take the pretriangulated DG subcategory $\dC\subset\prfdg V$ that is generated by the DG subcategories
$\mPhi(\prfdg X_1)$ and $\mPsi(\prfdg X_2).$ Lemma \ref{lemma_orthogonal} implies that there is a semi-orthogonal decomposition
\[
\Ho(\dC)\cong \langle \Phi(\prf X_1),\; \Psi(\prf X_2)\rangle.
\]
Since $\Phi$ and $\Psi$ are fully faithful,  Propositions \ref{gluing_semi-orthogonal} and \ref{gluing_quasifunctors}
give us that there are quasi-equivalences
\[
\dC\cong \mPhi(\prfdg X_1)\underset{\mS}{\oright}\mPsi(\prfdg X_2)
\cong
\prfdg X_1 \underset{\E^{\cdot}}{\oright}\prfdg X_2,
\]
where $\mS$ is the bimodule given by rule (\ref{s_bimodule}). By Theorem \ref{smooth_glue} the full DG subcategory $\dC\subset\prfdg V$
is smooth and proper as a gluing of smooth and proper DG categories via the perfect bimodule $\mS.$
Hence $\Ho(\dC)\cong \Ho(\prfdg X_1 \underset{\E^{\cdot}}{\oright}\prfdg X_2)$ is admissible in $\prf V.$
\hfill$\Box$

\begin{remark}
{\rm
It is useful to take in account that the category $\prf V$ from  Theorem \ref{main} has a semi-orthogonal decomposition of the form
$
\prf V=\langle \T_1,\dots \T_k\rangle
$
 such that each $\T_i$ is equivalent to one of the four categories, namely
$\prf\kk,\; \prf X_1,\; \prf X_2,$ and  $\prf (X_1\times X_2).$
It follows from the construction of $V$ as a two-step blowup of a projective space $\PP^N$ along
$P_1$ and $\wt{P}_2.$ By definition, $P_1$ and $P_2$ are projective bundles over $X_1$ and $X_2$ respectively and
$\wt{P}_2$ is the blowup of $P_2$ along $Z,$ where $Z$ is a sequence of projective bundles over $X_1\times X_2.$
}
\end{remark}

\subsection{Gluing of geometric noncommutative schemes}
In this section we extend results from the previous section to the case of geometric noncommutative schemes.
Actually all these statements are direct consequences of corresponding assertions for smooth and projective schemes.

Let $X_i,\; i=1,\dots,n$ be smooth and projective schemes. Let $\dN_i,\; i=1,\dots, n$ be  small pretriangulated DG categories.
Denote by $\N_i=\Ho(\dN_i)$ the homotopy triangulated categories.
Suppose that for all $i$ there are quasi-functors $\mF_i: \dN_i\to\prfdg X_i$
such that the induced exact functors $F_i: \N_i\to \prf X_i$ are fully faithful and have right and left adjoint functors.
 This means that $\N_i$ are admissible subcategories in $\prf X_i$ with respect to the full embeddings given by $F_i.$
This conditions imply that $\dN_i$ are geometric noncommutative schemes and, moreover, the DG categories $\dN_i$
are smooth and proper by Theorem \ref{smooth_glue}.

\begin{theorem}\label{main2}
Let DG categories $\dN_i,\; i=1,\dots, n$ and smooth projective schemes $X_i,\; i=1,\dots, n$ be as above.
Let $\dC$ be a proper pretriangulated DG category with full embeddings
 of DG categories $\dN_i\subset \dC$ such that $\C=\Ho(\dC)$ has a semi-orthogonal decomposition of the form
$
\C\cong\langle\N_1, \N_2,\dots, \N_n\rangle,
$
where $\N_i=\Ho(\dN_i).$ Then there are a smooth and projective scheme $V$ and  a quasi-functor $\mF: \dC\to\prfdg V$ such that
the induced functor $F: \C\to \prf V$ is fully faithful and has right and left adjoint functors, i.e. $\dC$ is
a geometric noncommutative scheme.
\end{theorem}
\begin{proof}
The case $n=1$ is evident.
Consider the main case $n=2.$
By Proposition \ref{gluing_semi-orthogonal} the DG category $\dC$ is quasi-equivalent
to a gluing of $\dN_1$ and $\dN_2$ via
a $\dN_2\hy\dN_1$\!--bimodule $\mS$ that is defined by the rule
\[
\mS (B, A)=\dHom_{\dC}(A, B), \quad \text{with}\quad A\in\dN_1 \;\text{and}\; B\in\dN_2.
\]
Since $\dC$ is proper the bimodule $\mS$ is a DG functor from $\dN_2\otimes\dN_1^{\op}$ to $\prfdg\kk$
and by Theorem \ref{dg_saturated} it is perfect, because
$\dN_i$ are smooth and proper.
By Proposition \ref{gluing_semi-orthogonal} the DG category $\dC$ is quasi-equivalent to the gluing
$\dN_1\underset{\mS}{\oright}\dN_2.$ By Theorem \ref{smooth_glue} we obtain that $\dC$ is smooth.

Consider quasi-functors $\mF_i: \dN_i\to \prfdg X_i.$ We know that $\mF_i$ establish quasi-equivalences with
enhancements of admissible subcategories in $\prf X_i.$
By Theorem \ref{toen}
the DG category of perfect
DG modules over $\prfdg X_2^{\op}\otimes_{\kk}\prfdg X_1$ is equivalent to $\prfdg(X_1\times X_2).$
Thus the quasi-functors $\mF_i$ induce the extension and induction quasi-functors
\[
(\mF_1\otimes\mF_2)^*:\prfdg (\dN_2^{\op}\otimes\dN_1)\to \prfdg(X_1\times X_2),
\quad
(\mF_1\otimes\mF_2)_*:\prfdg(X_1\times X_2)\to \prfdg (\dN_2^{\op}\otimes\dN_1)
\]
and by  Proposition \ref{geometric} the extension functor induces a fully faithful functor
between homotopy categories.
This implies that the bimodule $\mS$ is quasi-isomorphic to
a bimodule of the form $(\mF_1\otimes\mF_2)_* \E^{\cdot}$ for some perfect complex $\E^{\cdot}$ on
$X_1\times X_2.$

By Proposition \ref{gluing_quasifunctors} there is a quasi-functor
\[
\mF_1\underset{\phi}{\oright}\mF_2: \dN_1\underset{\mS}{\oright}\dN_2 \lto \prfdg X_1\underset{\E^{\cdot}}{\oright}\prfdg X_2,
\]
where $\phi$ is a quasi-isomorphism between $\mS$ and $(\mF_1\otimes\mF_2)_* \E^{\cdot}.$
Since  the functors $F_i: \N_i\to \prf X_i$ are fully faithful, the induced functor
\[
F_1\underset{\phi}{\oright}F_2: \Ho(\dN_1\underset{\mS}{\oright}\dN_2) \lto \Ho(\prfdg X_1\underset{\E^{\cdot}}{\oright}\prfdg X_2)
\]
is fully faithful too by Proposition \ref{gluing_quasifunctors}.

By Theorem \ref{main} the DG category $\prfdg X_1 \underset{\E^{\cdot}}{\oright}\prfdg X_2$ is quasi-equivalent to a full DG subcategory
of $\prfdg V$ for some smooth and projective scheme $V.$ Consider the composition quasi-functor
\[
\mF: \dC\stackrel{\sim}{\lto} \dN_1\underset{\mS}{\oright}\dN_2 \stackrel{\mF_1 \underset{\phi}{\oright}\mF_2}{\llto} \prfdg X_1\underset{\E^{\cdot}}{\oright}\prfdg X_2\lto \prfdg V.
\]
It induces an exact functor $\C\to\prf V$ that is fully faithful as a composition of fully faithful functors.
The DG category $\dC$ is proper and it is smooth as a gluing of smooth DG categories via a perfect DG bimodule.
Smoothness implies regularity of $\dC$ (see Proposition \ref{smooth_regular}).
Moreover, the category $\C$ is idempotent complete, because $\N_i$ are idempotent complete as admissible subcategories
 of $\prf X_i.$ Now, by Proposition \ref{admissible} regularity and properness of $\C$
give that the image of the fully faithful functor $F$ is an admissible
subcategory of $\prf V.$ Hence, $F$ admits  right and left adjoint functors.

The general case $n$ is done by induction. Denote by $\N'_2\subset \C$ the left orthogonal to $\N_1$ and denote by
$\dN'_2\subset \dC$ the full DG subcategory consisting of all objects from $\N'_2.$
We have semi-orthogonal decompositions $\N'_2=\langle \N_2,\dots N_n\rangle$ and $\C=\langle \N_1,\N_2\rangle.$
By the induction hypothesis, there are a smooth and projective scheme $V'$ and  a quasi-functor $\mF: \dN'_2\to\prfdg V'$ such that
the induced functor $F: \N'_2\to \prf V'$ is fully faithful and has right and left adjoint functors.
Now applying the proof for $n=2$ and the DG subcategories  $\dN_1$ and $\dN'_2$ in $\dC,$ we obtain the statement of the theorem for
$\dC.$
\end{proof}

\begin{corollary}\label{emb} Let $Y$ be a proper scheme over a field of characteristic $0.$ Then there are
a smooth projective scheme $V$ and a quasi-functor $\mF: \prfdg Y\to \prfdg V$ such that
the induced functor $F:\prf Y\to \prf V$ is fully faithful.
\end{corollary}
\begin{dok}
It follows from the main theorem of \cite[Th.1.4]{KL} that $Y$ has a so-called categorical resolution.
By the construction of this categorical resolution there is a quasi-functor from $\mG: \prfdg Y\to\dD,$
where $\dD$ is a gluing of DG categories of perfect complexes on smooth proper schemes, and
the induced functor $G:\prf Y\to \D$ is fully faithful.
Now as in Remark \ref{Hironaka} over a field of characteristic $0$ for any smooth proper scheme
there is a sequence of blowups with smooth centers such that the resulting smooth scheme is projective.
Hence $\dD$ is a gluing of geometric noncommutative schemes. By Theorem \ref{main2} there are a smooth projective $V$
and a quasi-functor
from $\dD$ to $\prfdg V$ which is fully faithful on homotopy categories. The composition of these  quasi-functors
gives us a quasi-functor $\mF: \prfdg Y\to \prfdg V$ that is also fully faithful on homotopy categories.
\end{dok}

\section{Application to finite algebras and exceptional collections}\label{applications}

\subsection{Finite dimensional algebras}

Let $\Lambda$ be a finite dimensional algebra over a base field $\kk.$
Denote by $\rd$ the (Jacobson) radical of $\Lambda.$ We know that $\rd^n=0$ for some $n.$
Define the index of nilpotency $i(\Lambda)$ of $\Lambda$ as the smallest integer $n$
such that $\rd^n=0.$

Let $S$ be the quotient algebra $\Lambda/\rd.$ It is semisimple and has only a finite number of simple non-isomorphic modules.
Denote by $\Md\Lambda$ and $\md\Lambda$ the abelian categories of all right modules and finite right modules over $\Lambda,$ respectively.

The following amazing result was proved by M.~Auslander.
\begin{theorem}\cite{Au}
Let $\Lambda$ be a finite dimensional algebra of index $n.$ Then the finite dimensional algebra
$\Gamma=\End(\bigoplus_{p=1}^{n} \Lambda/\rd^p)$ has the following properties:
\begin{enumerate}
\item[1)] $\gldim\Gamma\le n+1;$
\item[2)] there is a finite projective $\Gamma$\!--module $P$ such that $\End_{\Gamma}(P)\cong \Lambda.$
\end{enumerate}
\end{theorem}
Let us consider the bounded derived category of finite $\Gamma$\!--modules $\D^b(\md\Gamma).$
Since $\Gamma$ has finite global dimension, $\D^b(\md\Gamma)$ is equivalent to the category
of perfect complexes $\prf\Gamma.$
Some variants of the following theorem are known (see. e.g. \cite{KL}).

\begin{theorem}\label{algebra_exceptional}
Let $\Lambda$ be a finite dimensional algebra of index $n$ and let
$\Gamma=\End(\bigoplus_{p=1}^{n} \Lambda/\rd^p).$
The derived category $\prf \Gamma\cong \D^b(\md\Gamma)$ has a semi-orthogonal decomposition of the form
\[
\prf\Gamma=\langle\N_1,\dots,\N_n\rangle
\]
such that each subcategory $\N_i$ is semisimple, i.e. $\N_i\cong \langle K_i\rangle,$
where $K_i$ is semi-exceptional and for all $i$ the algebras $\End_{\Gamma}(K_i)$ are  quotients of the semisimple algebra
$S=\Lambda/\rd.$
\end{theorem}
\begin{dok}
Denote by $M$ the $\Lambda$\!--module $\bigoplus_{p=1}^{n} \Lambda/\rd^p$
and by $M_s$ the $\Lambda$\!--modules $\Lambda/\rd^s,\; s=1,\dots,n.$
Consider the functor $\Hom_{\Lambda}(M, -)$ from the abelian category
$\md\Lambda$ to the abelian category $\md\Gamma.$
Denote by $P_s$ the $\Gamma$\!--modules $\Hom_{\Lambda}(M, M_s).$
They are projective $\Gamma$\!--modules and $\Gamma=\bigoplus_{s=1}^n P_s.$

By Gabriel-Popescu theorem, since $M$ is a generator for $\Md\Lambda$ the functor $\Hom_{\Lambda}(M, -)$ from
$\Md\Lambda$ to $\Md\Gamma$ is fully faithful. Thus, there are isomorphisms
\[
\Hom_{\Gamma}(P_i, P_j)\cong \Hom_{\Lambda}(M_i, M_j)=\Hom_{\Lambda}(\Lambda/\rd^i, \Lambda/\rd^j)
\quad\text{for all}
\quad
1\le i, j\le n.
\]
Moreover, we have $\Hom_{\Lambda}(\Lambda/\rd^i, \Lambda/\rd^j)\cong \Lambda/\rd^j$ when $i\ge j.$
The canonical quotient morphisms $\Lambda/\rd^i\to \Lambda/\rd^j,$ when $i\ge j,$ induce morphisms
$\phi_{i,j}: P_i\to P_j.$

Let us consider $\phi_{i, i-1}$ and the induced exact triangles
\begin{equation}\label{semi_exceptional}
\xymatrix{
K_i\ar[r] & P_i \ar[r]^{\phi_{i, i-1}} & P_{i-1}\ar[r] & K_i[1], & i=2,\dots n
}
\end{equation}
in $\prf \Gamma.$ These triangles define objects $K_i$ for $i=2,\dots, n.$
We also set $K_1=P_1.$

Now, since $P_i$ are projective and $\Hom_{\Gamma}(P_i, P_j)\cong \Lambda/\rd^{j}$ when $i\ge j,$ we have vanishing
\begin{equation}\label{vanish}
\Hom_{\Gamma}(K_i, P_j[l])=0,\quad\text{for all}\quad l \quad\text{when}\quad i>j.
\end{equation}
Using definition (\ref{semi_exceptional}) of $K_i$ we immediately obtain semi-orthogonality conditions
\[
\Hom_{\Gamma}(K_i, K_j[l])=0, \quad\text{for all}\quad l \quad\text{when}\quad i>j.
\]
Finally, we have to compute $\bR\Hom_{\Gamma}(K_i, K_i)$ for all $i.$ Exact triangles (\ref{semi_exceptional}) give us
that the vector spaces $\Hom_{\Gamma}(K_i, K_i[l])$ are cohomology of the complexes

\begin{equation}\label{cohomology}
\Hom_{\Gamma}(P_{i-1}, P_i)\lto \Hom_{\Gamma}(P_{i}, P_i)\bigoplus \Hom_{\Gamma}(P_{i-1}, P_{i-1})
\lto \Hom_{\Gamma}(P_{i}, P_{i-1})
\end{equation}
that coincide with the complexes
\[
\Hom(\Lambda/\rd^{i-1}, \Lambda/\rd^i)\lto \Hom(\Lambda/\rd^{i}, \Lambda/\rd^i)\bigoplus \Hom(\Lambda/\rd^{i-1}, \Lambda/\rd^{i-1})
\lto \Hom(\Lambda/\rd^{i}, \Lambda/\rd^{i-1}).
\]
The morphism $\Hom(\Lambda/\rd^{i-1}, \Lambda/\rd^{i-1})
\to \Hom(\Lambda/\rd^{i}, \Lambda/\rd^{i-1})$ is an isomorphism and the morphism
$\Hom(\Lambda/\rd^{i-1}, \Lambda/\rd^i)\lto \Hom(\Lambda/\rd^{i}, \Lambda/\rd^i)$ is an injection.
This implies that the complex (\ref{cohomology}) has only zero cohomology.
Therefore,
\[
\Hom_{\Gamma}(K_i, K_i[l])=0, \quad\text{for all}\quad l\ne 0 \quad \text{and all}\quad i=1,\dots, n.
\]
Denote by $S_i$ the algebra of endomorphisms $\End_{\Gamma}(K_i),$ where $i=1,\dots, n.$
We know that $S_1=\End_{\Gamma}{P_1}\cong\End_{\Lambda}(\Lambda/\rd)=S$ is semisimple.

Let $a\in S_i$ be an element. It can be presented by a pair of morphisms $(a_{i}, a_{i-1})$ included in commutative diagram
\[
\begin{CD}
P_i & @>\phi_{i, i-1}>> & P_{i-1}\\
@Va_iVV && @VV a_{i-1}V\\
P_i & @>\phi_{i, i-1}>> & P_{i-1}
\end{CD}
\]
The vanishing conditions (\ref{vanish}) implies that the morphism $a_{i-1}$ is uniquely determined by $a_i.$
Thus, the element $a_i\in\Hom_{\Gamma}(P_i, P_i)\cong\Lambda/\rd^i$
induces an endomorphism of $K_i$ and we see that there is a homomorphism of algebras
 $\Lambda/\rd^{i}\to \End_{\Gamma}(K_i)$ that is surjective.
If now $a_i\in\End(P_i)=\Lambda/\rd^i$ belongs to $\rd,$  then as an endomorphism
of $\Lambda/\rd^i$ it sends $\rd^{i-1}$ to zero.
This implies that it is induced by a morphism of $\Lambda/\rd^{i-1}$ to $\Lambda/\rd^{i}.$
Thus, we obtain that the pair of morphisms $(a_i, a_{i-1})$ is induced by a morphism from $P_{i-1}$ to $P_i$
if $a_i\in \rd.$
This means that the algebra of endomorphisms $\End_{\Gamma}(K_i)$ is a quotient of the semisimple algebra $S=\Lambda/\rd.$
Therefore, the algebras $S_i=\End_{\Gamma}(K_i)$ are semisimple for all $i=1,\dots, n$ too.
As $P_i,\;i=1,\dots,n$ generate the category $\prf\Gamma$ the objects $K_i,\;i=1,\dots,n$ generate
$\prf\Gamma$ as well, and we obtain a semi-orthogonal decomposition
\[
\prf\Gamma=\langle\langle K_1\rangle,\cdots,\langle K_n\rangle\rangle,
\]
where all $K_i$ are semi-exceptional and $S_i=\End_{\Gamma}(K_i)$ are quotients of the algebra $S=\Lambda/\rd.$
\end{dok}

Consider now the $\Gamma$\!--module $P_n=\Hom_{\Lambda}(M, \Lambda),$ where
 $M=\bigoplus_{p=1}^{n} \Lambda/\rd^p.$ It is projective and $\End_{\Gamma}(P_n)\cong\Lambda.$ This object gives us two functors
\[
(-)\otimes_{\Lambda} P_n: \prf\Lambda\to\prf\Gamma\quad\text{and}\quad
\Hom_{\Gamma}(P_n, -):\D^b(\md\Gamma)\to\D^b(\md\Lambda)
\]
The first functor is fully faithful while the second functor is a quotient. Since $\Gamma$ has  finite global dimension
there is an equivalence $\prf\Gamma\cong\D^b(\md\Gamma).$
Now if $\Lambda$  also has finite global dimension, then the second functor is right adjoint to the first one and the
category $\prf\Lambda$ is right admissible in $\prf\Gamma$ with respect to the full embedding $(-)\otimes_{\Lambda} P_n.$

Recall that a semisimple algebra $S$ over a field $\kk$ is called separable over $\kk$ if it is a projective
$S^{\op}\otimes_{\kk} S$\!--module. It is well-known that a semisimple algebra $S$ is separable if
it is a direct sum of simple algebras, the centers of which are separable extensions of $\kk.$

\begin{theorem}\label{algebra}
Let $\Lambda$ be a finite dimensional algebra over $\kk.$ Assume that $S=\Lambda/\rd$ is a separable $\kk$\!--algebra.
Then there are a smooth projective scheme
$V$ and a perfect complex $\E^{\cdot}$ such that $\End(\E^{\cdot})\cong\Lambda$ and $\Hom(\E^{\cdot}, \E^{\cdot}[l])=0$
for all $l\ne 0.$
\end{theorem}
\begin{dok}
As above, let $\Gamma=\End(\bigoplus_{p=1}^{n} \Lambda/\rd^p).$
Consider the DG category $\prfdg\Gamma.$ By Theorem
\ref{algebra_exceptional} the triangulated category $\prf\Gamma$ has
a semi-exceptional collection $(K_1,\dots, K_n)$ and
\[
\prf\Gamma=\langle\N_1,\dots,\N_n\rangle
\]
where $\N_i=\langle K_i\rangle$ is semisimple. Thus, for any $i$
the object $K_i$ is a direct sum of the form $\bigoplus_{j=1}^{m_i} K_{ij},$ where $K_{ij}$ are completely orthogonal to each other
for fixed $i$ and different $j.$
Moreover, each $\End_{\Gamma}(K_{ij})$ is a simple algebra, i.e it is a matrix algebra over a division $\kk$\!--algebra $D_{ij}.$
By assumption $S$ is separable. Hence, all $\End_{\Gamma}(K_{ij})$ are separable as quotients of $S.$
Thus we obtain that the centers $\kk_{ij}$ of all $D_{ij}$ are separable extensions of $\kk.$

Now as in Example \ref{Severi-Brauer} we can consider a Severi-Brauer variety $SB(D_{ij})$ that is a smooth projective scheme over $\kk_{ij}$
and over $\kk$ too, because $\kk_{ij}\supset\kk$ is a finite separable extension. It was mentioned in Example \ref{Severi-Brauer}
that there is a vector bundle $E_{ij}$ on $SB(D_{ij})$ such that it is w-exceptional and $\End(E_{ij})\cong D_{ij}.$
This implies that each DG category $\prfdg D_{ij}$ is a full DG subcategory of the DG category  $\prfdg SB(D_{ij}),$
 and $SB(D_{ij})$ is smooth and projective over $\kk.$

All categories $\N_i$ have  complete orthogonal decompositions of the form
$\N_i=\N_{i1}\oplus\cdots\oplus \N_{i m_i},$ where $\N_{ij}=\langle K_{ij}\rangle$ are equivalent to
 $\prf D_{ij}.$ These decompositions induce a semi-orthogonal decomposition for $\prf\Gamma$ of the form
\[
\prf\Gamma=\langle\N_{11}, \N_{12},\dots,\N_{1 m_1},\N_{21},\dots, \N_{n m_n}\rangle.
\]
Applying Theorem \ref{main2} we obtain that there are a smooth projective scheme $V$ and a quasi-functor
from $\mF: \prfdg\Gamma\to \prfdg V$ such that the homotopy functor
$F:\prf\Gamma\to\prf V$ is fully faithful and establishes an equivalence with an admissible subcategory
in $\prf V.$ Denote by $\E^{\cdot}$ the perfect complex $F(P_n),$ where $P_n=\Hom_{\Lambda}(M, \Lambda)$
is a projective $\Gamma$\!--module. Since $F$ is fully faithful we have isomorphisms
\[
\Hom_V(\E^{\cdot}, \E^{\cdot}[l])\cong\Hom_{\Gamma}(P_n, P_n[l]).
\]
When $l\ne 0$ it is $0,$ and it is isomorphic to the algebra $\Lambda$ for $l=0.$
\end{dok}

\begin{corollary}\label{algebra_inclusion}
Let $\Lambda$ be a finite dimensional algebra over $\kk$ for which $S=\Lambda/\rd$ is separable $\kk$\!--algebra.
Then there is a smooth projective scheme $V$ such that
the DG category $\prfdg\Lambda$ is quasi-equivalent to a full DG subcategory of $\prfdg V.$ Moreover, if $\Lambda$ has  finite global dimension, then
$\prf\Lambda$ is  admissible in $\prf V.$
\end{corollary}
\begin{dok}
By Theorem \ref{algebra} there is a smooth projective scheme $V$ and
a perfect complex $\E^{\cdot}$ such that
$\End(\E^{\cdot})\cong\Lambda$ and $\Hom(\E^{\cdot},
\E^{\cdot}[l])=0$ for all $l\ne 0.$ Hence, the DG algebra
$\dHom_{\prfdg V}(\E^{\cdot}, \E^{\cdot})$ is quasi-isomorphic to
the algebra $\Lambda.$ Thus, by Proposition \ref{Keller2} there is a
quasi-functor $\mF:\prfdg\Lambda\to \prfdg V$ induced by the
embedding of $\dHom_{\prfdg V}(\E^{\cdot}, \E^{\cdot})$ into $\prfdg
V$ such that the homotopy functor $F: \prf\Lambda\to \prf V$ is
fully faithful. If $\Lambda$ has  finite global dimension, then the
category $\prf\Lambda$ is regular and proper. Hence it is admissible
in $\prf V$ by Proposition \ref{admissible}.
\end{dok}

\begin{remark}
{\rm
Note that over a perfect field all semisimple algebras are separable.
Thus, if $\kk$ is perfect, then results of this section apply  to all finite dimensional algebras.
}
\end{remark}

\begin{remark}
{\rm Theorem \ref{algebra} tells us, in particular, that for any
finite dimensional algebra $\Lambda$ of finite global dimension the
category $\prf\Lambda$ can be embedded to a triangulated category
with a full semi-exceptional collection (actually, with
w-exceptional collection).
On the other hand, as was pointed out to me  by Theo Raedschelders
there are  finite dimensional algebras of finite global dimension for which
the category of perfect complexes does not have a full exceptional collection
(actually, it does not have any exceptional object).
Such examples were discussed by Dieter Happel in \cite{Hap}.
This gives a counterexample to Jordan-H\"older property for triangulated categories of perfect complexes on smooth projective schemes.
More precisely,  there are admissible subcategories $\T$ in
$\prf X$ on smooth projective variety $X$ such that $\prf X$ has a full exceptional
collection but $\T$ does not have an exceptional object at all.
Another example coming from a quiver was constructed by Alexei Bondal
and were discussed by Alexander Kuznetsov in \cite{Kuz} from geometric point of view.
 }
\end{remark}

\subsection{Exceptional collections}\label{exceptional_coll}

In this section we describe a more useful procedure of constructing a scheme that admits a full exceptional collection and contains
as a subcollection an exceptional collection  given in advance.

Let $\dA$ be a small smooth and proper pretriangulated DG category over a field $\kk$ such that the homotopy category
$\Ho(\dA)$ has a semi-orthogonal decomposition of the form
\[
\Ho(\dA)=\langle \N,\; \langle E\rangle\rangle,
\]
where $\N$ is a full admissible subcategory and $E$ is an exceptional object, i.e.
$\langle E\rangle\cong \prf\kk.$

Assume that the enhancement of $\N$ induced from $\dA$ is
quasi-equivalent to  a full DG subcategory  $\dN\subset \prfdg X$
for a smooth and projective irreducible scheme $X$ such that $H^i(X,
\cO_X)=0$ for all $i>0,$ i.e. the structure sheaf $\cO_X$ is
exceptional.
\begin{remark}
{\rm
The assumption $H^i(X, \cO_X)=0$ for all $i>0$ is not restrictive. Indeed, for any smooth and projective scheme
we can consider a closed immersion  into a projective space $\PP^N$ for some large $N.$
Take the blowup $Z$ of $\PP^N$ along $X.$ Then $\prfdg X$ is quasi-equivalent to a full DG subcategory
in $\prfdg Z$ and  $H^i(Z, \cO_{Z})=0$ for all $i>0.$ Now we can take $Z$ instead of $X.$
}
\end{remark}

We have that $\N\cong\Ho(\dN)$ is an admissible subcategory in $\prf X.$
By Propositions \ref{gluing_semi-orthogonal} and \ref{gluing_quasifunctors} the DG category $\dA$ is quasi-equivalent to a gluing of $\dN$ and $\prfdg\kk$
via some $\dN$\!--module $\mS.$
Since $\prfdg X$ and $\dN$ are saturated by Theorem \ref{dg_saturated}, the DG module $\mS$ can be represented by a perfect
complex $\cS^{\cdot}$ on $X,$ i.e the DG $(\prfdg X)$\!--module $\dHom(-,\; \cS^{\cdot})$ after restriction of $\dN$
is quasi-isomorphic to the DG $\dN$\!--module $\mS.$

Shifting the complex $\cS^{\cdot}$ by $[m]$ for an appropriate  $m\in\ZZ$ we can suppose that
\[
\cS^{\cdot}\cong\{\cS^0\to \cS^1\to\cdots\to \cS^k\},
\]
where all $\cS^i$ are vector bundles on $X.$

By Proposition \ref{line} there is a smooth morphism $f: Z\to X$ and a line bundle $\L$
on $Z$ such that $\bR f_* \L\cong \cS^{\cdot}$ and $\bR f_*\L^{-1}=0.$
Moreover, the morphism $f$ is a sequence of projective bundles. Hence $\bR f_*\cO_Z\cong \cO_X$
and the inverse image functor $f^*:\prf X \to\prf Z$ is fully faithful.
 Since
$\bR f_*\L^{-1}=0$ we have
\[
\Hom_Z(\L, \; f^* A)\cong \Hom_Z(\cO_Z, \; f^* A\otimes \L^{-1})\cong \Hom_X(\cO_X, \;  A\otimes\bR f_*\L^{-1})=0
\]
for any $A\in\prf X.$ Therefore $f^*(\prf X)$ is in the right orthogonal $\langle \L\rangle ^{\perp}.$

Since $\cO_X$ is exceptional the structure sheaf
$\cO_Z$ and any line bundle on $Z$ are also exceptional. Therefore, the admissible subcategory
 $\T\subset \prf Z$ which is generated by $f^*(\N)$ and $\L$ has a semi-orthogonal decomposition of the form
$
\T\cong\langle \N, \prf\kk\rangle,
$
Denote by $\dT$ enhancement of  $\T$ the induced from $\prfdg Z.$  By Propositions \ref{gluing_semi-orthogonal} and \ref{gluing_quasifunctors}
 the DG category $\dT$ is quasi-equivalent to $\dA$ because both of them
are quasi-equivalent to the gluing $\dN\underset{\mS}{\oright}\prfdg\kk$ via $\mS.$

The procedure described above can be considered as an induction step in the proof of the following theorem
 while the base case is the point $\Spec\kk.$ Thus, we obtain.
\begin{theorem}\label{exc_col}
Let $\dA$ be a small DG category over $\kk$ such that the homotopy category
$\Ho(\dA)$ has a full exceptional collection
\[
\Ho(\dA)=\langle E_1,\dots, E_n\rangle.
\]
Then there are a smooth projective scheme $X$ and an exceptional collection of line bundles
$\sigma=(\L_1,\dots, \L_n)$ on $X$ such that the DG subcategory of $\prfdg X,$
generated by $\sigma,$ is quasi-equivalent to $\dA.$
Moreover, $X$ is a sequence of projective bundles and has a full exceptional collection.
\end{theorem}
\begin{remark}
{\rm
The scheme $X$ has a full exceptional collection as a sequence of projective bundles
(see \cite{Blow}). Furthermore, it follows from construction that a full exceptional collection on $X$ can be chosen in a way that
it contains the collection $\sigma=(\L_1,\dots, \L_n)$ as a subcollection.
}
\end{remark}

\subsection{Noncommutative projective planes}

In this section we consider a particular case of noncommutative projective planes, in sense of noncommutative deformations of the
usual projective plane, and
present explicit embeddings of categories of perfect complexes on them to categories of perfect complexes on smooth projective commutative schemes.

Noncommutative deformations of the projective plane have been described in \cite{ATV}.
The category $\prf \PP^2$ has a full exceptional collection $(\cO, \cO(1), \cO(2)).$ Note also that mirror symmetry relations for
noncommutative planes is described in \cite{AKO}.

Any deformation of the category $\prf \PP^2$ is a category with three ordered objects
$F_0, F_1, F_2$
and with three-dimensional spaces of homomorphisms from  $F_i$ to $F_j$ when $j-i=1$ and a six-dimensional
vector space as  Hom from $F_0$ to $F_2.$
 Any such category  is determined by
the composition tensor $\mu: V\otimes U\to W,$ where $\dim V=\dim U=3$ and
$\dim W=6.$ This map should be surjective. Denote by $T$ the kernel of $\mu$ and by $\nu: T\to  V\otimes U.$
We will consider only the nondegenerate (geometric)
case, where the restrictions $\nu_{u^*}: T\to V$ and
$\nu_{v*}: T\to U$ have rank at least two for all nonzero elements
$u^*\in U^{\vee}$ and $v^*\in V^{\vee}.$
The equations $\det \nu_{u^*}=0$ and $\det\nu_{v^*}=0$ define closed subschemes
$\Gamma_U\subset\PP(U^{\vee})$ and $\Gamma_V\subset\PP(V^{\vee}).$
Namely, up to projectivization the set of points of $\Gamma_U$
(resp.\ $\Gamma_V$) consists of all
$u^*\in U^{\vee}$ (resp.\ $v^*\in V^{\vee}$) for which the rank of
$\nu_{u^*}$ (resp.\ $\nu_{v*}$) is equal to $2.$
It is easy to see that the correspondence which associates  the kernel of the map
$\nu_{v^*}^{\vee}: U^{\vee}\to T^{\vee}$ to
a vector $v^*\in V^{\vee}$ defines
an isomorphism between $\Gamma_V$ and $\Gamma_U$.
Moreover, under these circumstances $\Gamma_V$ is either the entire
projective plane $\PP(V^{\vee})$ or a cubic
in $\PP(V^{\vee}).$ If $\Gamma_V=\PP(V^{\vee}),$
then $\mu$ is isomorphic to the tensor $V\otimes V\to S^2 V,$
i.e. we get the usual projective plane $\PP^2.$

Thus, the non-trivial case is the situation, where $\Gamma_V$ is a cubic, which we will denote by $E$.
This curve comes equipped with two embeddings into the
projective planes $\PP(U^{\vee})$ and $\PP(V^{\vee})$ respectively; by restriction of
$\cO(1)$ these embeddings
determine two line bundles $\L_1$ and $\L_2$ of degree $3$ on $E$,
and it can be checked that $\L_1\ne \L_2.$ This construction has an inverse:

\begin{construction}\label{constr:mu}
{\rm
The tensor $\mu$ can be reconstructed from the triple
$(E, \L_1, \L_2).$
Namely, the spaces $U, V$ are isomorphic to
$H^0(E, \L_1)$ and $H^0(E, \L_2)$ respectively, and the tensor
$\mu: V\otimes U\to W$ is nothing but the canonical map
$
H^0(E, \L_2) \otimes H^0(E, \L_1)\lto H^0(E, \L_2\otimes\L_1).
$
}
\end{construction}

\begin{remark}\label{rmk:comm}
{\rm
Note that  we can also consider  a triple $(E, \L_1, \L_2)$ such that
$\L_1\cong \L_2.$ Then the procedure described above produces
a tensor with $\Gamma_V\cong \PP(V^{\vee}),$ which defines the usual
commutative projective plane. In this case the tensor  $\mu$ does not depend on the curve $E.$
The details of these constructions and statements can be found in \cite{ATV}.
}
\end{remark}

Now let us see what our construction gives in the case of noncommutative planes.
In some sense we repeat the construction from the proof of Proposition \ref{line} in this case.
The subcategory generated by $(F_0, F_1)$ is a subcategory
of $(\cO, \cO(1))$ on the usual $\PP^2=\PP(U^{\vee}).$ Now we should glue to this category
the object $F_2.$ The projection of $F_2$ on the subcategory generated by
$(F_0, F_1)$ can be represented by the complex
\begin{equation}
\label{resol}
T\otimes\cO\stackrel{\nu}{\lto} V\otimes\cO(1)
\end{equation}
on $\PP^2.$ This complex  is a resolution of the cokernel of this map. It is isomorphic
to the sheaf $\cO_{E}(\L_1\otimes \L_2),$
where $E$ is a curve of degree 3 on $\PP^2,$  $\L_1$ is the restriction of $\cO(1)$
on $E,$ and $\L_2$ is another line bundle of degree 3 on $E.$

At first, we take the projectivization of $V\otimes\cO(1).$ We obtain $\PP(U^{\vee})\times\PP(V^{\vee})$
and the line bundle $\cO(1,1)$ on it.
The direct image of this bundle on the first component is isomorphic to $V\otimes\cO(1)$ on $\PP(U^{\vee}).$
After that we consider $Y=\PP(U^{\vee})\times\PP(V^{\vee})\times \PP^1$ and the line bundle
$\cO(1, 1, -2)$ on it. The morphism $\nu$ induces an element
$\epsilon\in \Ext^1_Y(T\otimes\cO_Y,\; \cO(1,1,-2)).$
Now we take a vector bundle $\F$ on $Y$ that is extension
\begin{equation}\label{extension}
0\lto \cO(1,1,-2)\lto\F\lto T\otimes\cO_Y\lto 0.
\end{equation}
Finally, we take $Z=\PP(\F)$ and the line bundle $\L=\cO_Z(1)$.
The direct image of $\L$ with respect to the projection on
$\PP(U^{\vee})$ is isomorphic to the complex (\ref{resol}).  Now, if we consider
three line bundles $\cO_Z,$ the pull back of $\cO(1)$ from $\PP(U^{\vee}),$
and $\L=\cO_Z(1)$ on $Z,$
then it is an exceptional collection on $Z$
and the corresponding subcategory in $\prf Z,$ generated by them, is equivalent
to the category of perfect complexes on the noncommutative projective plane.
 Different noncommutative projective planes correspond to the different vector bundles
 $\F$ that depend on the element $\epsilon.$

\begin{proposition}
For any noncommutative deformation of the projective plane $\PP^2_{\mu}$ the DG category
$\prfdg \PP^2_{\mu}$ is quasi-equivalent to a full DG subcategory
of $\prfdg Z,$ where $Z$ is the projectivization  of a 4-dimensional vector bundle $\F,$
defined as extension (\ref{extension}), over $Y=\PP^2\times\PP^2\times \PP^1.$
\end{proposition}

New results on realizations of the categories $\prf \PP^2_{\mu}$ on noncommutative projective planes
$\PP^2_{\mu}$ as admissible subcategories of the categories on smooth projective varieties can be found in \cite{O3}.

\end{document}